\input ./style/arxiv-general.cfg
\documentclass[aop,MSNbibl,dvips]{arximspdf}
\makeatletter
   \@ifpackageloaded{graphicx}{}{\usepackage{graphicx}}
\makeatother
\usepackage{mathbh}

\doi{10.1214/14-AOP991}
\volume{44}
\issue{2}
\pubyear{2016}
\firstpage{883}
\lastpage{923}
\docsubty{FLA}

\makeatletter

\newcommand{\rrvert}{\vert}
\newcommand{\llvert}{\vert}
\newcommand{\mathds}{\mathbh}
\newcommand{\eqref}[1]{(\ref{#1})}
\newtheorem{prop}{Proposition}[section]
\newproclaim{rem}[prop]{Remark}
\newtheorem{lem}[prop]{Lemma}
\newtheorem{theo}[prop]{Theorem}
\newtheorem{coro}[prop]{Corollary}
    \newtheorem{conj}[prop]{Conjecture}
\newproclaim{definition}[prop]{Definition}
\newcommand{\eps}{\varepsilon}
\newcommand{\Pb}{\mathbb{P}}
\newcommand{\Fk}{\mathcal{F}_k}
\newcommand{\Mc}{\mathcal{M}}
\newcommand{\E}{\mathbb{E}}
\newcommand{\om}{\omega}

\newcommand{\kc}{\widehat{k}}
\newcommand{\kt}{\widetilde{k}}
\newcommand{\kb}{\bar{k}}
\newcommand{\ks}{\underline{k}}
\newcommand{\al}{\alpha}
\newcommand{\1}{\mathbh{1}}
\makeatother

\begin{document}
\begin{frontmatter}

\title{Stuck walks: A conjecture of Erschler, T\'oth and Werner}
\runtitle{Stuck walks: A conjecture of Erschler, T\'oth and Werner}

\begin{aug}
\author[A]{\fnms{Daniel}~\snm{Kious}\corref{}\ead[label=e1]{daniel.kious@epfl.ch}}
\runauthor{D. Kious}
\affiliation{Ecole Polytechnique F\'ed\'erale de Lausanne}
\address[A]{
EPFL SB MATHAA PRST\\
MA B1 537 (B\^{a}timent MA) \\
Station 8\\
CH-1015 Lausanne\\
Switzerland\\
\printead{e1}}
\end{aug}

%
\received{\smonth{9} \syear{2013}}
\revised{\smonth{11} \syear{2014}}

\begin{abstract}
In this paper, we work on a class of self-interacting nearest neighbor random walks, introduced in
[\textit{Probab. Theory Related Fields} \textbf{154} (2012) 149--163], for which there is competition between repulsion of neighboring edges and attraction of next-to-neighboring edges. Erschler,
T\'oth and Werner proved in
[\textit{Probab. Theory Related Fields} \textbf{154} (2012) 149--163] that, for any $L\ge1$, if the parameter $\alpha$ belongs to a certain interval $(\alpha_{L+1},\alpha_L)$, then such random walks localize on $L+2$
sites with positive probability. They also conjectured that this is the almost sure behavior. We prove this conjecture
partially, stating that the walk localizes on $L+2$ or $L+3$ sites almost surely, under the same assumptions. We also
prove that, if $\alpha\in(1,+\infty)=(\alpha_2,\alpha_1)$, then the walk localizes a.s. on $3$ sites.
\end{abstract}

\begin{keyword}[class=AMS]
\kwd[Primary ]{60K35}
\kwd[; secondary ]{60G20}
\kwd{60G42}
\end{keyword}
\begin{keyword}
\kwd{Stuck walks}
\kwd{reinforced random walks}
\kwd{localization}
\kwd{Rubin}
\kwd{time-line construction}
\kwd{martingale}
\end{keyword}
\end{frontmatter}


\section{Introduction}


Let $X:=(X_n)_{n\ge0}$ be a nearest neighbor walk on the integer lattice $\mathbb{Z}$. Let $l_k(j)$ be the local time
on the nonoriented edge $\{j-1,j\}$ up to time~$k$:
\begin{eqnarray*}
l_k(j)&:=&\sum_{m=1}^k
\1_{\{\{X_{m-1},X_m\}=\{j-1,j\}\}}.
\end{eqnarray*}

Define the filtration $(\Fk)_{k\in\mathbb{N}}$ generated by the process, that is,  for all
$k\in\mathbb{N}$, $\Fk=\sigma(X_0,\ldots,X_k)$.

Fix a real parameter $\alpha$ and define the following linear combination of local times on neighboring and
next-to-neighboring edges of a site $j$:
\begin{equation}
\Delta_k(j):=-\alpha l_k(j-1)+l_k(j)-l_k(j+1)+
\alpha l_k(j+2),\label{defdelta}
\end{equation}
and, with a slight abuse of notation,
\[
\Delta_k:=\Delta_k(X_k).
\]
Fix another real parameter $\beta>0$. In this paper, we consider the walk introduced in \cite{ETW}, defined by $X_0=0$
and the following conditional transition probabilities:
\begin{eqnarray}
\Pb(X_{k+1}=X_k\pm1|\Fk)&=& \frac{e^{\pm\beta\Delta_k}}{e^{-\beta\Delta_k}+e^{\beta\Delta_k}}.\label{ec1}
\end{eqnarray}

The linear combination $\Delta_k$ can be seen as the \emph{local stream} felt by the walker. When this stream is
positive (resp., negative), the walk bends toward the right (resp., toward the left). The value of the parameter $\beta$
does not affect very much the behavior of the walk, whereas the value of $\alpha$ plays a crucial role, as we explain
below.

This model is a generalization of the true self-repelling walk (TSRW) in one dimension. We can recover the TSRW with edge repulsion, by choosing $\alpha=0$, as well as the TSRW with site repulsion, by\vspace*{1pt} choosing $\alpha=-1$. In the case of edge repulsion, a nondegenerate scaling limit for $X_k/k^{2/3}$ is proved in \cite{BT9} and the same scaling limit is conjectured for site repulsion; we refer the reader to \mbox{\cite{APP,BT10}} for more details.
Some interesting work has also been done concerning the continuous
space--time true self-repelling motion, see, for instance, \cite{LD,LDBT,BTWW}.

A scaling behavior similar to the one of the edge-repelled TSRW is expected for $\alpha\in[-1,1/3)$. Roughly speaking,
$\alpha$ is then sufficiently close to $0$ and this can be seen as a perturbation of the TSRW. The case $\alpha=1/3$ is more mysterious, and the nondegenerate scaling of the walk might be like $k^{2/5}$. For
$\alpha\in(-\infty,-1)$, the walker is repelled by its neighboring edges and even more strongly by its
next-to-neighboring edges. In this case, it seems that the walk self-builds trapping environments, which causes a
slowing down of the walk. We refer the reader to \cite{ETW2} for more detailed discussions and arguments.

In \cite{ETW}, Erschler, T\'oth and Werner focus on the case where $\alpha>1/3$. In particular, the walker is repelled
by its neighboring edges and attracted by its next-to-neighboring edges. Therefore, there is competition between
repulsion and \mbox{attraction}: the last one might win, resulting in localization of the walk on an arbitrarily large
interval depending on $\al$.

More precisely, let us define subintervals of $(1/3,+\infty)$ corresponding, as we will see later, to different
possible features of the walk.

Define the sequence $(\alpha_L)_{L\ge1}$ by $\alpha_1:=+\infty$ and for all $L\ge2$:
\begin{equation}
\label{defalpha} \alpha_L:=\frac{1}{1+2\cos({2\pi}/(L+2))},
\end{equation}
so that this sequence decreases from $+\infty$ to $1/3$ as $L$ increases from $1$ to $+\infty$.

Define, for all $k\in\mathbb{N}$ and $j\in\mathbb{Z}$, the number of visits to the site $j$, up to time $k$:
\[
Z_k(j):=\sum_{m=1}^k
\1_{\{X_m=j\}}=\frac{l_k(j)+l_k(j+1)+\1_{\{X_k=j\}}-\1_{\{j=0\}}}{2},
\]
and let $Z_\infty(j)$ be its limit when $k$ goes to infinity.
Let $R$ (resp., $R'$) be the set of points that are visited at least once (resp., infinitely often), that is,
\begin{eqnarray*}
R&:=& \bigl\{j\in\mathbb{Z}\dvtx Z_{\infty}(j)>0 \bigr\},
\\
R'&:=& \bigl\{j\in\mathbb{Z}\dvtx Z_{\infty}(j)=\infty \bigr
\}.
\end{eqnarray*}

In \cite{ETW}, the authors prove the following result.

\begin{theo}[(Erschler, T\'oth and Werner \cite{ETW})] \label{ETW1}
Suppose that $L\ge1$. We have:
\begin{itemize}
\item[--] If $\alpha<\alpha_L$, then, almost surely, $\llvert  R'\rrvert  \ge L+2$, or $R'=\varnothing$.

\item[--] If $\alpha\in(\alpha_{L+1},\alpha_L)$, then the probability that $\llvert  R'\rrvert  =L+2$ is positive.
\end{itemize}

Moreover, there exists a deterministic real valued vector $(v_1,\ldots,v_{L+1})$, such that, on the event
$\{R'=\{x,x+1,\ldots,x+L+1\}\}$, we almost surely have the following law of large numbers:
\[
\lim_{k\rightarrow +\infty} \frac{1}{k}\bigl(l_k(x+1),
\ldots,l_k(x+L+1)\bigr)=(v_1,\ldots,v_{L+1}).
\]
\end{theo}

\begin{rem}
As in \cite{ETW}, we will give an explicit form for the vector $(v_1,\ldots,v_{L+1})$ later in this paper: it corresponds to the solutions of a linear system given by Proposition~\ref{systpos} (in the case $K=L$).
\end{rem}

Erschler, T\'oth and Werner also propose the following conjecture.

\begin{conj}[(Erschler, T\'oth and Werner \cite{ETW})]
If $\alpha\in(\alpha_{L+1},\alpha_L)$, then $\llvert  R'\rrvert  =L+2$ almost surely.
\end{conj}

The main goal of this paper is to prove the following result, which partially settles the conjecture.

\begin{theo} \label{L2L3}
Assume that $\al\in(\al_{L+1},\al_L)$, then the walk localizes on $L+2$ or $L+3$ sites almost surely, that is,
$\llvert  R'\rrvert  \in\{L+2,L+3\}$ a.s.
\end{theo}

We also give, in both cases, the asymptotic behavior of the local times, which correspond to those obtained in
\cite{ETW}; see Proposition~\ref{systpos}.

When $\al>1$, that is, $\al\in(\al_2,\al_1)$, we can improve this result and prove the conjecture in this case, which is
done in Section~\ref{palgrd}.

\begin{theo} \label{algrd}
Assume that $\al\in(1,+\infty)$, then the walk localizes on $3$ sites almost surely, that is, $\llvert  R'\rrvert  =3$ a.s.
\end{theo}

Besides, we also believe that, for general $\al>1/3$ and $L\ge1$ such that $\alpha\in(\alpha_{L+1},\alpha_L)$, the event $\llvert  R'\rrvert  =L+3$ does not occur as it corresponds to an unstable
limiting behavior.

Another problem is the one of the critical values of $\al$, that is not treated here, neither in \cite{ETW}.
Nevertheless, Erschler, T\'oth and Werner give good arguments to make us believe that, if $\alpha=\alpha_{L+1}$, then the walk will almost surely not be stuck on $L+2$ sites; see the concluding remarks of \cite{ETW}.

Note also that it is usually quite challenging to obtain an almost sure behavior for a reinforced random walk and let
us mention some interesting results (see also \cite{12086375,VLPT08,MR2282181}). Tarr\`es proved in \cite{PT5pts} the localization on $5$ sites of the
vertex-reinforced random walk (VRRW) with linear reinforcement, which is an important result in the field of reinforced
random walks. In \cite{PTsurvey}, Tarr\`es proposed another proof of this result, introducing a variant of
so-called Rubin's construction, allowing powerful couplings. This variant was also used by Basdevant, Schapira and
Singh who proved in \cite{BSS1} a phase transition of the behavior of VRRWs, for nondecreasing weight functions of
order $n\log\log n$. The same authors, in \cite{BSS2}, characterize the behavior of VRRWs with sublinear nondecreasing weight functions, thanks
to some index (which is an integer that depends on the weight function) and they prove an a.s. lower bound for the size of the localization set (which is not sharp), depending on
this index. They propose a conjecture with better bounds, which they prove with positive probability. In particular,
they exhibit walks that localize on arbitrarily large intervals, as we do in the present paper. The major difference between the behavior of Stuck walks and VRRWs is that in the case of VRRWs on $\mathbb{Z}$, only $3$ vertices are visited during a positive proportion of the total time, and the other infinitely often visited sites are seldom visited, no matter how many they are.


\section{Sketch of the proof of Theorem \texorpdfstring{\protect\ref{L2L3}}{1.4}} \label{mainproof}


Let us describe the techniques used to prove Theorem \ref{L2L3}. First, as in \cite{ETW}, we compare the local times of
the walk to the solutions of a linear system.
This linear system is not easy to handle generally but we prove additional results on the solutions of this system, in
order to generalize some results of \cite{ETW} and to emphasize, through trigonometric identities, instability
properties of some vertices that will force the walk to localize on one of their sides.
Finally, we adapt a variant of Rubin's construction, introduced in \cite{PTsurvey}, to a class of
\emph{nonmonotonic} weight functions.

Fix $L\ge1$, and assume that $\al\in(\al_{L+1},\al_L)$.
Let us first state the following proposition, proved in Section~\ref{finiterange}.

\begin{prop}\label{proprang}
The walk almost surely visits finitely many sites, that is, $|R|<+\infty$ a.s.
\end{prop}

Note that this last proposition discards the possibility of transience of the walk, that is, $R'\neq\varnothing$ a.s.

Knowing from \cite{ETW} that $\llvert  R'\rrvert  \ge L+2$ a.s., our next goal is to prove that $\llvert  R'\rrvert  \le L+3$.
Then, let $x\in\mathbb{Z}$ be the leftmost infinitely visited site, and let $K\in\mathbb{N}\cup\{0\}$ be the number of
\emph{interior} lattice sites of $R'$, that is,
\begin{eqnarray*}
x&:=&\inf \bigl\{y\in\mathbb{Z}\dvtx y\in R' \bigr\},
\\
K&:=& \bigl\llvert R' \bigr\rrvert -2.
\end{eqnarray*}

Note that $x$ and $K$ are random variables but that they can take only countably many values, so we can fix
$x\in\mathbb{Z}$ and $K\in\mathbb{N}\cup\{0\}$ and work on the event $ \{R'=\{x,x+1,\ldots,x+K+1\} \}=\{x=\inf
R'\}\cap\{\llvert  R'\rrvert  =K+2\}$.

The following result, proved in Section~\ref{aswei}, is crucial.

\begin{prop}\label{lemdeltak1}
Let $x\in\mathbb{Z}$ and let $K\in\mathbb{N}$. Almost surely, we have
\[
\bigl\{R'=\{x,\ldots,x+K+1\} \bigr\} \subset \bigcap
_{j=x+1}^{x+K} \biggl\{\lim_{k\rightarrow+\infty}
\frac{\Delta_k(j)}{k}=0 \biggr\}.
\]
\end{prop}

The last proposition is a generalization of Proposition~2 of \cite{ETW}, which would only give this result for
$K\le L$, which is not sufficient. Hence, Proposition \ref{lemdeltak1} requires a more technical proof, which needs, in
particular, results of Section~\ref{large}.

From Proposition \ref{lemdeltak1}, we know that the normalized local times eventually approach the set of solutions
$(l_0,\ldots,l_{K+2})$ of the linear system defined by
\begin{equation}
d_1=d_2=\cdots=d_K=0\quad\mbox{and}\quad
\sum_{j=1}^{K+1}l_j=1,\label{tag}
\end{equation}
where, for all $j\in\{1,\ldots,K\}$, $d_j=-\al l_{j-1}+l_j-l_{j+1}+\al l_{j+2}$, and with the extra conditions
$l_0=l_{K+2}=0$.

So, all the information we can get about this system gives us information about the asymptotic behavior of the local
times. The sole purpose of Section~\ref{trigogo} is to study this generalized Fibonacci sequence and its solutions.
In fact, we also use some properties of the solutions of this system to prove Propositions \ref{proprang}~and~\ref{lemdeltak1}.

Let us roughly describe the solutions of such a system, which is detailed in Section~\ref{trigogo}. First, define
$d_0=-l_1+\al l_2$ and $d_{K+1}=-\al l_{K}+l_{K+1}$.
\begin{itemize}
\item If $K<L$, then the solution to the system is unique, all the $l_j$'s are positive, $d_0$ is negative and, by symmetry,
$d_{K+1}=-d_0>0$. As $\Delta_k(x)/k$ and $\Delta_k(x+K+1)/k$, respectively, approach $d_0$ and $d_{K+1}$ as $k$ goes to
infinity, it gives us the intuition that we cannot have $\llvert  R'\rrvert  <L+2$, since otherwise the local streams at the boundaries
would strongly push the walker out of this interval.

\item If $K=L$, then the solution is still unique, with all the $l_j$'s positive, and we have $d_0=-d_{L+1}>0$. So, we guess
that $\llvert  R'\rrvert  =L+2$ is a good candidate, as the local streams on the boundaries would keep the walker inside the interval.

\item If $K=L+1$, the unique solution is similar to the previous one, which makes $\llvert  R'\rrvert  =L+3$ another good candidate.

\item Otherwise, if $K>L+1$, the solution may not be unique, and we cannot determine {a priori} the
sign of $d_0$ nor $d_{K+1}$. Moreover, as it is noticed in \cite{ETW} (see also Remark \ref{remsyst}), we could find
many $K$'s for which the set of solutions to the system is such that all the $l_j$'s are nonnegative, $d_0>0$ and
$d_{K+1}<0$, that is, the local streams on the boundaries push the walker inward.
In other words, we could find many good candidates for the size of $R'$. The goal is then to exclude these larger
values of $K$.
\end{itemize}

Recall that $L\ge1$ and that $\al\in(\al_{L+1},\al_L)$, so that $\al>1/3$. Let $\om\in(0,\pi)$ be the unique real number such that
\[
\cos(\om)=\frac{1-\al}{2\al}.
\]
The following corollary is a bit less general than Proposition~\ref{system} which we state and prove in Section~\ref{large}. As we have already noticed, the linear system is quite
difficult to handle when $K$ is large, but we can prove some useful results under some convenient assumptions.

\begin{coro}\label{systneg}
Assume that $\al\in(\al_{L+1},\al_L)$ and $K\geq L$. If $(l_0,\ldots,l_{K+2})$ satisfies the previous system \eqref{tag} and if
$l_1,\ldots,l_{K+1}\geq0$, then:

\begin{longlist}[(ii)]
\item[(i)] $d_0\ge c(K)$ and $d_{K+1}\leq -c(K)$, where $c(K)$ is a positive constant depending only on $\al$ and $K$;

\item[(ii)] for the same positive constant $c(K)$, we have
\begin{eqnarray*}
l_{L+2}-\al l_{L+1}&=&-\frac{\sin((({L+2})/{2})\om)}{\sin(({L}/{2})\om)}l_1+2\al
\frac{\cos({\om}/{2})\sin(((L+3)/{2})\om)}{\sin(({L}/{2})/\om)}l_{L+1}
\\
&\le& -c(K).
\end{eqnarray*}
\end{longlist}
\end{coro}

The first point is used to prove Proposition \ref{lemdeltak1}, whereas the second point is used to prove the following
series convergence.

\begin{lem}\label{somfinie}
Let $x,K\in\mathbb{Z}$ such that $K\ge L$. Then, almost surely, for any $a>0$,
\[
\bigl\{R'=\{x,\ldots,x+K+1\} \bigr\}\subset \Biggl\{\sum
_{k=1}^{+\infty} e^{a\beta[l_k(x+L+2)-\al
l_k(x+L+1)]}<+\infty \Biggr\}.
\]
\end{lem}

Its proof is simple but needs several technical details to be rigorous, thus it is set forth in Section~\ref{aswei2}: we use that, on the event $\{R'=\{x,\ldots,x+K+1\}\}$,  the vector of local times $(l_k(x),\ldots,l_k(K+2))/k$ approaches, as $k$ goes to infinity, the set of solutions to the
linear system \eqref{tag} introduced in the previous paragraph.

In order to prove Theorem \ref{L2L3}, we make use of an embedding of the original walk into a continuous-time process, obtained via a time-line construction similar to the one often used for Markov chains. It was initially proposed by Herman Rubin, and used in the papers of Davis~\cite{Davis} and Sellke~\cite{Sellke}, in the context of edge-reinforced random walks. We make use of a variant of this construction, initially introduced by Tarr\`es~\cite{PTsurvey} for the study of VRRW with nondecreasing \emph{weight functions}, which allows powerful couplings in the one-dimensional case, in order to prove nonconvergence toward some unstable limit sets.

However, the nonmonotonic setting studied here does not satisfy the conditions of the results in \cite{PTsurvey}, and we therefore extend the techniques to that case: this will enable us to eliminate the large unstable intervals as candidates for the localization set.

Let us take advantage of some different ways to write the probability for the walker to jump on his left when he is on a
certain vertex:
\begin{eqnarray}\label{ec3}
&& \Pb(X_{k+1}=X_k-1|\Fk)\nonumber
\\[-2pt]
&&\qquad = \frac{1}{1+e^{2\beta\Delta_k}}
\\[-2pt]
&&\qquad = \frac{e^{2\beta [-l_k(X_k)+\alpha l_k(X_k-1) ]}}{e^{2\beta [-l_k(X_k)+\alpha
l_k(X_k-1) ]}+e^{2\beta [-l_k(X_k+1)+\alpha l_k(X_k+2) ]}}.\nonumber
\end{eqnarray}

It is worth noticing that \eqref{ec3} consists of two terms: one depending only on the local times on edges on the
left of the considered vertex and the second one depending only on the local times on edges on the right of it. Hence,
after one excursion on the right, or on the left, of this vertex, only one of the two terms changes. Note
that we still do not have any monotonicity.

Before making this construction more precise, let us introduce some definitions. First, for all $y\in\mathbb{Z}$ and
$k\in\mathbb{N}$, define the number of times the \emph{oriented} edge \mbox{$(y,y\pm1)$} has been crossed, up to time $k$:
\[
N_k(y,y\pm1):=\sum_{n=0}^{k-1}
\1_{\{X_n=y,X_{n+1}=y\pm1\}},
\]
and notice that if $X_k=y$ then we have
\begin{eqnarray*}
Z_k(y\pm1)&=&\frac{l_k(y+((1\pm1)/2))+l_k(y+((1\pm3)/2)-\1_{\{y\pm1=0\}}}{2},
\\
N_k(y,y\pm1)&=&\frac{l_k(y+((1\pm1)/2))-\1_{\{\pm y<0\}}}{2}.
\end{eqnarray*}
Define two functions $f^+$ and $f^-$ from $\mathbb{Z}\times(\mathbb{N}\cup\{0\})$ in $\mathbb{R}^+$, and a function $w$
from $\mathbb{N}\cup\{0\}$ to $\mathbb{R}^+$, such that, for all $n\in\mathbb{N}\cup\{0\}$ and $y\in\mathbb{Z}$,
\begin{eqnarray}
f^{\pm}(y,n)&:=&\exp \bigl(2\beta \bigl[2(1+\al)n-\al
\1_{\{y\pm1=0\}}+(1+\al)\1_{\{\pm y<0\}
} \bigr] \bigr),\label{deffpm}
\\[-2pt]
w(n)&:=&\exp (4\beta\al n ).\label{defw}
\end{eqnarray}

Then, if $X_k=y$, we have
\begin{eqnarray}
e^{2\beta [-l_k(y+1)+\alpha l_k(y+2) ]}&=& \frac{w(Z_k(y+1))}{f^+(y,N_k(y,y+1))}, \label{egrub1}
\\[-2pt]
e^{2\beta [-l_k(y)+\alpha l_k(y-1) ]}&=& \frac{w(Z_k(y-1))}{f^-(y,N_k(y,y-1))},\label{egrub2}
\end{eqnarray}
which enables us to rewrite  \eqref{ec3} as
\begin{eqnarray}\label{probarubin}
&&\Pb(X_{k+1}=X_k-1|\Fk)
\nonumber
\\[-2pt]
&&\qquad = \frac{w(Z_k(X_k -1))}{f^-(X_k,N_k(X_k,X_k-1))}
\\[-2pt]
&&\quad\qquad{}\Big/
\biggl({\frac{w(Z_k(X_k
-1))}{f^-(X_k,N_k(X_k,X_k-1))}+ \frac{w(Z_k(X_k+1))}{f^+(X_k,N_k(X_k,X_k+1))}}\biggr). \nonumber
\end{eqnarray}

Now, we are going to use Rubin's construction, which is the following. Recall that, in our case, the function $w$
is defined as in \eqref{defw}, but this construction can be done for any nondecreasing weight function. The way we
present this construction is very similar to the one in \cite{BSS1}. We will embed the walk into a continuous-time
process, involving time-lines corresponding to
sequences of clocks (see Figure~\ref{etwdessin1}). More precisely, fix a collection of positive real numbers:
\[
\xi:=\bigl(\xi^{\pm}_{k}(y),y\in\mathbb{Z},k\ge0\bigr)\in
\mathbb{R}_+^{\mathbb{N}}.
\]
We call $\Mc$ the function which maps a collection of positive real numbers $\xi$ to a continuous-time walk
$\widetilde{X}=(\widetilde{X}_t)_{t\in\mathbb{R}_+}=:\Mc(\xi)$ on $\mathbb{Z}$, constructed as follows.

\begin{figure}[t]

\includegraphics{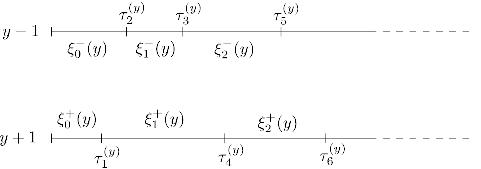}

   \caption{Time-lines around a vertex $y$. The times $\tau^{(y)}_k$ are times of jumps  from $y$ to $y\pm 1$.}
\label{etwdessin1}
\end{figure}

For each $y\in\mathbb{Z}$ and $k\in\mathbb{N}\cup\{0\}$, the value $\xi_k^+(y)$ [resp., $\xi_k^-(y)$] will be related to
the duration of a clock attached to the oriented edge $(y,y+1)$ [resp., $(y,y-1)$], and given the collection $\xi$, the
evolution of the walk is deterministic, created as follows:
\begin{itemize}
\item Set $\widetilde{X}_0=0$ and attach two clocks to the oriented edges $(0,-1)$ and $(0,1)$ ringing, respectively, at
times $\xi_0^-(0)/w(0)$ and $\xi_0^+(0)/w(0)$.
\item At time $\tau_1:= (\xi_0^-(0)/w(0) )\wedge (\xi_0^+(0)/w(0) )$, one of the alarms rings, then
the walker crosses instantaneously the corresponding edge and both clocks are stopped. Then we have set
$\widetilde{X}_{\tau_1}=\pm1$ depending on which clock has rung first. If both clocks ring at the same time, then
$\widetilde{X}$ stays at $0$ forever and we say that the construction fails.
\end{itemize}
Now, assume that we have constructed $\widetilde{X}$ up to some time $t>0$ at which time the process makes a right jump
from $y-1$ to $y$, for some vertex $y\in\mathbb{Z}$. Denote by $k$ the number of jumps from $y$ to $y-1$ and by $n$ the
number of visits to $y-1$ before time $t$. We continue the construction according to the following procedure:
\begin{itemize}
  \item Start a new clock attached to the oriented edge $(y,y-1)$, which will ring after a time $\xi_k^-(y)/w(n)$.
\item If the walker has already been in $y$ previously, then restart the clock attached to the oriented edge $(y,y+1)$
that has been stopped (without ringing) the last time the process left $y$. Otherwise, start a new clock for $(y,y+1)$
which will ring after a time $\xi_0^+(y)/w(0)$.
\item When one of the alarms rings, the walker crosses instantaneously the corresponding edge, and both clocks are
stopped. If both alarms ring simultaneously, then $\widetilde{X}$ stays at $y$ forever and we say that the construction
fails.
\end{itemize}
We follow the symmetric procedure when the process jumps from $y+1$ to $y$.

Let us naturally adopt the notations $\widetilde{\mathcal{F}}_t$, $\tilde{l}_t(y)$, $\widetilde{N}_t(y,y\pm1)$,
$\widetilde{Z}_t(y)$, $\widetilde{R}'$ inherited from $\mathcal{F}_k$, $l_k(y)$, $N_k(y,y\pm1)$, $Z_k(y)$, $R'$.

Let us define $\tau_0:=0$ and $\tau_k$ the time\vspace*{1pt} of the $k$th jump of $\widetilde{X}$. Then define the discrete-time
\emph{embedded walk} $(\widetilde{X}_{\tau_k})_k$.

Recalling the definitions \eqref{deffpm} of $f^\pm$,
\eqref{defw} of $w$, and the transition probability \eqref{probarubin}, we prove the following proposition.

\begin{prop} \label{rubexp}
Choose $\xi$ to be such that, for all $y\in\mathbb{Z}$ and $k\in\mathbb{N}\cup\{0\}$, $\xi_k^\pm(y)$ are exponential
random variables, independent from each other, and with mean $f^\pm(y,k)$. Define\vspace*{1pt} the continuous-time walk
$\widetilde{X}:=\Mc(\xi)$ and its embedded walk $(\widetilde{X}_{\tau_k})$.

Then\vspace*{1pt} the construction of $\widetilde{X}$ does not fail with probability $1$. Moreover, the processes
$(\widetilde{X}_{\tau_k})_{k\ge0}$ and $(X_k)_{k\ge0}$ have the same distribution.
\end{prop}

\begin{pf}
Recall that if $U$ and $V$ are two independent exponential random variables of parameter $u$ and $v$ (i.e., with mean
$1/u$ and $1/v$), then $\Pb(U<V)=u/(u+v)$ and $\Pb(U=V)=0$.

Then it suffices to recall \eqref{probarubin} and to notice that, conditionally on $\widetilde{\mathcal{F}}_{\tau_k}$,
the quantities
$\xi_{\widetilde{N}_{\tau_k}(\widetilde{X}_{\tau_k},\widetilde{X}_{\tau_k}\pm1)}^\pm(y)/w(\widetilde{Z}_{\tau_k}(\widetilde{X}_{\tau_k}\pm1))$
are independent exponential random variables with respective means
\[
\frac{f^+(\widetilde{X}_{\tau_k},\widetilde{N}_{\tau_k}(\widetilde{X}_{\tau_k},\widetilde{X}_{\tau_k}+1))}{w(\widetilde{Z}_{\tau_k}(\widetilde{X}_{\tau_k}+1))} \quad\mbox{and}\quad \frac{f^-(\widetilde{X}_{\tau_k},\widetilde{N}_{\tau_k}(\widetilde{X}_{\tau_k},\widetilde{X}_{\tau_k}-1))}{w(\widetilde{Z}_{\tau_k}(\widetilde{X}_{\tau_k}-1))}.
\]\upqed
\end{pf}

Until\vspace*{1pt} the end of this section, we work with $\xi$ defined as in this last proposition, which then implies that
$(\widetilde{X}_{\tau_k})_k$, $(\widetilde{Z}_{\tau_k}(y))_{k,y}$, $(\tilde{l}_{\tau_k}(y))_{k,y}$ and $\widetilde{R}'$
have the same laws as $(X_k)$, $(Z_k(y))$, $(l_k(y))$ and $R'$.

Define, for each $y\in\mathbb{Z}$, the total time consumed by
the clocks attached to the oriented edge $(y,y\pm1)$:
\begin{equation}
T_y^\pm:=\sum_{k\ge0}
\1_{\{\widetilde{X}_{\tau_k}=y,\widetilde{X}_{\tau_{k+1}}=y\pm1\}}\frac{\xi^{\pm}_{\widetilde{N}_{\tau_k}(y,y\pm1)}(y)}{w(\widetilde{Z}_{\tau_k}(y\pm1))}.\label{Ty}
\end{equation}

From the time-line construction, it is clear that
\begin{eqnarray}\label{inclusion}
&& \bigl\{\widetilde{Z}_\infty(y-1)=\infty \bigr\}\cap \bigl\{
\widetilde{Z}_\infty(y+1)=\infty \bigr\}\cap \bigl\{T_y^+
\wedge T_y^-<\infty \bigr\}
\nonumber\\[-8pt]\\[-8pt]\nonumber
&&\qquad \subset \bigl\{T_y^+=T_y^-<\infty \bigr\}
\end{eqnarray}
since otherwise, after a certain time, the walk would jump infinitely many times on one side of $y$ before performing
one more jump on the other side.

Section~\ref{rubin} is dedicated to proving the next proposition, using monotonicity properties of the time-line
construction.

\begin{prop} \label{proba0}
If $\widetilde{X}$ is a continuous-time process defined as in Proposition \ref{rubexp}, then, for any $y\in\mathbb{Z}$,
\[
\Pb \bigl(T_y^+=T_y^-<\infty,\widetilde{Z}_\infty(y-1)=
\widetilde{Z}_\infty(y+1)=\infty \bigr)=0.
\]
\end{prop}

The following proposition is proved in Section~\ref{martarg}, using martingale arguments.

\begin{prop}\label{inclu2}
If $\widetilde{X}$ is a continuous-time process defined as in Proposition \ref{rubexp}, then, for any $y\in\mathbb{Z}$,\begin{eqnarray*}
\Biggl\{\sum_{k=0}^{+\infty} e^{2\beta[\tilde{l}_{\tau_k}(y)-\al
\tilde{l}_{\tau_k}(y-1)]}<
\infty \Biggr\}\subset \bigl\{T_y^-<\infty \bigr\}.
\end{eqnarray*}
\end{prop}

Before proving all the previous technical results in the following sections, let us use them in order to prove the main theorem.

\begin{pf*}{Proof of Theorem \ref{L2L3}}
Fix $x\in\mathbb{Z}$ and fix $K\ge L+2$, so that $\{\widetilde{R}'=\{x,\ldots,x+K+1\}\}\subset\{\llvert  \widetilde{R}' \rrvert  \ge
L+4\}$ and
\begin{eqnarray*}
&& \bigl\{\widetilde{R}'=\{x,\ldots,x+K+1\} \bigr\}
\\
&&\qquad \subset \bigl\{
\widetilde{Z}_\infty(x+L+1)=\infty \bigr\}\cap \bigl\{
\widetilde{Z}_\infty(x+L+3)= \infty \bigr\}.
\end{eqnarray*}

Using the series convergence of Lemma \ref{somfinie} and Proposition \ref{inclu2}, we deduce
\[
\bigl\{\widetilde{R}'=\{x,\ldots,x+K+1\}\bigr\}\subset\bigl
\{T_{x+L+2}^-<+\infty\bigr\}.
\]
Hence, using the inclusion \eqref{inclusion},
\[
\bigl\{\widetilde{R}'=\{x,\ldots,x+K+1\}\bigr\}\subset \bigl
\{T_{x+L+2}^-=T_{x+L+2}^+<+\infty\bigr\}.
\]
Therefore, using Proposition \ref{proba0}, it implies that
\[
\Pb \bigl(\widetilde{R}'=\{x,\ldots,x+K+1\} \bigr)=0.
\]
Taking finally the union over $x\in\mathbb{Z}$ and $K\ge L+2$, and using Proposition \ref{rubexp}, we conclude that
\[
\Pb \bigl(\bigl\llvert R'\bigr\rrvert \ge L+4 \bigr)=\Pb \bigl(
\bigl\llvert \widetilde{R}' \bigr\rrvert \ge L+4 \bigr)=0.
\]\upqed
\end{pf*}


\section{Monotonicity of the time-line construction and martingale arguments} \label{monomart}


In Section~\ref{rubin}, we prove Proposition \ref{proba0} by taking advantage of some monotonicity properties occurring
in the time-line construction. The idea of the whole proof is due to Tarr\`es, \cite{PTsurvey}, but we use, at the end of this
proof, a more convenient argument, due to Basdevant, Schapira and Singh, \cite{BSS1}.
In Section~\ref{martarg}, we use classical martingale arguments to prove Proposition \ref{inclu2}.

In these two sections, we generalize some previous results to a class of dynamics involving
nonmonotonic weight function. More precisely, all the results hold for any positive functions
$f^{\pm}(\cdot,\cdot)$, and any positive nondecreasing function $w(\cdot)$. Note that $w/f^{\pm}$ can be nonmonotonic.


\subsection{Proof of Proposition \texorpdfstring{\protect\ref{proba0}}{2.6}} \label{rubin}


Recall the time-line construction introduced in the previous section, in which we have defined the function $\Mc$ which
maps a collection $\xi$ of positive real numbers to a continuous-time walk $\Mc(\xi)$.

\begin{definition}
Given $\xi^1=((\xi^1)_k^\pm(y))_{y\in\mathbb{Z},k\ge0}$ and $\xi^2=((\xi^2)_k^\pm(y))_{y\in\mathbb{Z},k\ge0}$ two
collections of random variables on $\mathbb{R}_+$, we say that $\xi^1\gg \xi^2$ if, for all $k\ge0$ and
$y\in\mathbb{Z}$, we have that $(\xi^1)_k^+(y)\le(\xi^2)_k^+(y)$ and $(\xi^1)_k^-(y)\ge(\xi^2)_k^-(y)$ almost surely.
\end{definition}

Let us define two continuous-time walks $\widetilde{X}{}^1:=\Mc(\xi^1)$
and $\widetilde{X}{}^2:=\Mc(\xi^2)$, adopting the natural notation $\widetilde{Z}_t^1(y)$ and $\widetilde{Z}_t^2(y)$,
for any $y\in\mathbb{Z}$. For any $i\in\mathbb{N}$ and any $y\in\mathbb{Z}$, let $t_{\{y,y+1\}}^1(i)$ [resp.,
$t_{\{y,y+1\}}^2(i)$] be the time of the $i$th crossing of $\widetilde{X}{}^1$ (resp., $\widetilde{X}{}^2$) of the
nonoriented edge $\{y,y+1\}$, with the convention that $t_{\{y,y+1\}}^1(0)=t_{\{y,y+1\}}^2(0)=0$.

\begin{lem}[(Tarr\`es, \cite{PTsurvey})] \label{lemPT}
Assume that $\xi^1\gg\xi^2$ and that the constructions of $\widetilde{X}{}^1$ and $\widetilde{X}{}^2$ do not fail, then,
for any $i\in\mathbb{N}$ and $y\in\mathbb{Z}$, we almost surely have that
\[
\widetilde{Z}_{t_{\{y,y+1\}}^1(i)}^1(y+1)\ge \widetilde{Z}_{t_{\{y,y+1\}}^2(i)}^2(y+1)
\]
and
\[
\widetilde{Z}_{t_{\{y,y+1\}}^1(i)}^1(y)\le
\widetilde{Z}_{t_{\{y,y+1\}}^2(i)}^2(y).
\]
\end{lem}

Using this last result, here follows the main proof of this section.

\begin{pf*}{Proof of Proposition \ref{proba0}}
In order to apply the last lemma, fix $y\in\mathbb{Z}$ and define the collection:
\[
\mathcal{H}_y:= \bigl(\bigl(\xi_k^\pm(x),k
\ge0,x\neq y\bigr),\bigl(\xi_k^-(y),k\ge 0\bigr),\bigl(
\xi_k^+(y),k\ge1\bigr) \bigr),
\]
where the $\xi$'s are independent and defined as in Proposition \ref{rubexp}, that is, $\xi_k^\pm(x)$ is an exponential
random variable with mean $f^\pm(x,k)$.

Then note that the process $\widetilde{X}$ defined in Proposition \ref{rubexp} has the same law as
$\Mc(\mathcal{H}_y,\xi_0^+(y))$, where $\xi_0^+(y)$ is an exponential random variable with mean $1$, independent from
$\mathcal{H}_y$. This means that $X$ has the same law as the embedded walk of $\Mc(\mathcal{H}_y,\xi_0^+(y))$.

Moreover, we can create a family of walks $\widetilde{X}{}^{(u)}:=\Mc(\mathcal{H}_y,u)$, $u>0$, which correspond to
continuous-time walks defined according to the time-line construction, with $\xi_0^+(y)=u$.

For all $u>0$, in a way similar to \eqref{Ty}, define
\begin{eqnarray*}
T_y^\pm(\mathcal{H}_y,u)&=&\sum
_{k\ge0} \1_{\{\widetilde{X}_{\tau_k}^{(u)}=y,\widetilde{X}_{\tau_k+1}^{(u)}=y\pm1\}}\frac{\xi^{\pm}_{\widetilde{N}_{\tau_k}^{(u)}(y,y\pm1)}(y)}{w(\widetilde{Z}_{\tau_k}^{(u)}(y\pm1))}
\\
&=&\sum_{k=0}^{\widetilde{N}_\infty^{(u)}(y,y\pm1)}\frac{\xi_k^\pm(y)}{w (\widetilde{Z}_{t_{\{y,y\pm1\}}^{(u)}(2k+\1_{\{\pm
y<0\}})}^{(u)}(y\pm1) )},
\end{eqnarray*}
which is the total time consumed by the clocks attached to the oriented edge \mbox{$(y,y\pm1)$} for the walk $X^{(u)}$.

Then, Lemma \ref{lemPT} implies that, for any $u,u'>0$ such that $u>u'$, almost surely on the event
$\{T_y^\pm(\mathcal{H}_y,u),T_y^\pm(\mathcal{H}_y,u')<\infty\}\cap\{\widetilde{Z}_\infty^{(u)}(y\pm1)=\widetilde{Z}_\infty^{(u')}(y\pm1)=\infty\}$, on which the times $T_y^\pm(\mathcal{H}_y,u),T_y^\pm(\mathcal{H}_y,u')$ are all four finite,
we have that
\[
T_y^+(\mathcal{H}_y,u)>T_y^+\bigl(
\mathcal{H}_y,u'\bigr)\quad\mbox{and}\quad
T_y^-(\mathcal{H}_y,u)\le T_y^-\bigl(
\mathcal{H}_y,u'\bigr).
\]
Thus, it implies that, given $\mathcal{H}_y$, there exists at most one value $u_0=u_0(\mathcal{H}_y)$ of $\xi_0^+(y)$
such that the event $\{T_y^+=T_y^-<\infty\}\cap\{\widetilde{Z}_\infty(y+1)=\widetilde{Z}_\infty(y-1)=\infty\}$ occurs.
If such a value does not exist, then we use the convention $u_0=\infty$.

Note that $\{\widetilde{Z}_\infty(y+1)=\widetilde{Z}_\infty(y-1)=\infty\}$ implies that the construction succeed, which
is already proved to happen a.s.

Then we conclude with
\begin{eqnarray*}
&& \Pb \bigl(T_y^+=T_y^-<\infty,\widetilde{Z}_\infty(y+1)=
\widetilde{Z}_\infty(y-1)=\infty \bigr)
\\
&&\qquad \le\E \bigl(\Pb \bigl(\xi_0^+(y)=u_0|
\mathcal{H}_y \bigr) \bigr)=0.
\end{eqnarray*}\upqed
\end{pf*}


\subsection{Proof of Proposition \texorpdfstring{\protect\ref{inclu2}}{2.7}} \label{martarg}


Fix $y\in\mathbb{Z}$ and assume for simplicity that $y>0$ (but similar arguments hold whenever $y\le0$). Let
$\theta_k$ be the time of the $k$th jump from $y-1$ to $y$, and let $i_k$ be the number of visits to site $y-1$ up to time
$\theta_k$. Define a new filtration $\mathcal{G}$ such that, for all $n\ge0$, we have
\[
\mathcal{G}_n=\sigma \bigl(\bigl(\xi_k^\pm(x),k
\ge0,x\neq y\bigr),\bigl(\xi_k^+(y),k\ge0\bigr),\bigl(
\xi_k^-(y),0\le k\le n\bigr) \bigr),\vadjust{\goodbreak}
\]
 and recall that the $\xi$'s are exponential clocks defined as in Proposition \ref{rubexp}. For any $n\ge0$, define
\begin{eqnarray*}
Y_y^-(n)&:=&\sum_{k=0}^n
\1_{\{\theta_{k+1}<\infty\}}\frac{f^-(y,k)}{w(i_{k+1})},
\\
T_y^-(n)&:=&\sum_{k=0}^n
\1_{\{\theta_{k+1}<\infty\}}\frac{\xi_{k}^-(y)}{w(i_{k+1})},
\\
M_n&:=&T_y^-(n)-Y_y^-(n),
\\
\langle M\rangle_n&:=&\sum_{k=0}^{n}
\E \bigl((M_{k+1}-M_k)^2|\mathcal{G}_k
\bigr).
\end{eqnarray*}

Notice that, given $\mathcal{G}_n$, we know the duration of the $n+1$ first clocks attached to the oriented edge
$(y,y-1)$, and the duration of any other clock attached to any other edge.

Then, as $y>0$, $\theta_{n+2}$, and thus $i_{n+2}$ are $\mathcal{G}_n$-measurable, whereas $\xi_{n+1}^-(y)$ is
independent of $\mathcal{G}_n$. We have
\begin{eqnarray*}
\E (M_{n+1}-M_n|\mathcal{G}_n )&=&
\frac{\1_{\{\theta_{n+2}<\infty\}}}{w(i_{n+2})}\E \bigl(\xi_{n+1}^-(y)-f^-(y,n+1) \bigr)=0,
\\
\E \bigl((M_{n+1}-M_n)^2|\mathcal{G}_n
\bigr)&=&\frac{\1_{\{\theta_{n+2}<\infty\}}}{ (w(i_{n+2}) )^2}\E \bigl( \bigl(\xi_{n+1}^-(y)-f^-(y,n+1)
\bigr)^2 \bigr)
\\
&=&\frac{\1_{\{\theta_{n+2}<\infty\}}(f^-(y,n+1))^2}{ (w(i_{n+2}) )^2}.
\end{eqnarray*}
Hence, $M_n$ is a square integrable $\mathcal{G}$-martingale. By a classic result about square integrable martingales, we know that $M_n$ converges a.s. to
a finite limit on the event $\{\langle M\rangle_\infty<\infty\}$, see \cite{Neveu}, Proposition VII-2.3. Thus, on the
event $\{\langle M\rangle_\infty<\infty\}\cap\{Y_y^-(\infty)<\infty\}$, we have a.s. $T_y^-=\lim_nT_y^-(n)<\infty$.

Besides, \eqref{egrub2} implies
\begin{eqnarray*}
Y_y^-(\infty)&\le& \sum_{k=0}^{+\infty}
e^{2\beta[\tilde{l}_{\tau_k}(y)-\al \tilde{l}_{\tau_k}(y-1)]},
\\
\langle M\rangle_\infty&\le&\sum_{k=0}^{+\infty}
e^{4\beta[\tilde{l}_{\tau_k}(y)-\al \tilde{l}_{\tau_k}(y-1)]},
\end{eqnarray*}
so that $\{Y^-_y(\infty)<\infty\}\subset\{\langle M\rangle_\infty<\infty\}$.
Thus, the following inclusion a.s.\break holds:
\begin{eqnarray*}
\Biggl\{\sum_{k=0}^{+\infty} e^{2\beta[\tilde{l}_{\tau_k}(y)-\al
\tilde{l}_{\tau_k}(y-1)]}<
\infty \Biggr\}\subset \bigl\{T_y^-<\infty \bigr\}.
\end{eqnarray*}

\begin{rem}
Note that, in fact, we do not need the explicit form of $f^\pm$ and~$w$. More generally, for any positive function
$f^\pm$ and $w$, the inclusions $\{Y^-_y(\infty)<\infty\}\subset\{\langle M\rangle_\infty<\infty\}$ and consequently
$\{Y^-_y(\infty)<\infty\}\subset\{T^-_y<\infty\}$ hold.
\end{rem}


\section{Trigonometric results} \label{trigogo}


\subsection{Definition of the linear system} \label{trigo}


Fix $L\ge1$ and assume that $\al\in(\al_{L+1}, \al_L)$, where $(\al_L)_{L\ge1}$ is defined in \eqref{defalpha}.\vadjust{\goodbreak}

We will describe some properties of the set of possible solutions $(l_0,
\ldots,\break l_{K+2})$ to linear equations that
are part of a bi-infinite Fibonacci-type sequence. First, we introduce this linear system, then we distinguish between
three cases, depending on the size of the system.

For all $K\in\mathbb{N}$ and for all $j\in\{1,\ldots,K\}$, let
\begin{eqnarray}
d_0&=&l_0-l_1 +\al l_2,
\nonumber
\\
d_j&=&-\alpha l_{j-1}+l_j-l_{j+1}+
\alpha l_{j+2},\label{dj}
\\
d_{K+1}&=&-\al l_k+l_{K+1}-l_{K+2}.
\nonumber
\end{eqnarray}

For all $K\in\mathbb{N}$, define the linear system:
%
{\renewcommand{\theequation}{$E_K$}
\begin{equation}
\label{eq}
d_1=d_2=
\cdots=d_K=0,\qquad l_0=0\quad\mbox{and}\quad \sum
_{j=1}^{K+1} l_j=1.
\end{equation}\setcounter{equation}{13}}%

Recall that, given $\al>1/3$, $\om\in(0,\pi)$ is the unique real
number such that
\[
\cos(\om)=\frac{1-\al}{2\al}.
\]
Then, using \eqref{defalpha},
\[
\frac{2\pi}{L+3}<\om<\frac{2\pi}{L+2}.
\]

In Section~\ref{small}, we describe the unique solution of the linear system \eqref{eq} with the extra condition
$l_{K+2}=0$, when $K$ is \emph{small}, namely $K\in\{0,\ldots,L+1\}$.

In Section~\ref{medium}, we prove some properties of an associated affine system when $K=L$. In other words, we
consider the same system, but with general $d_j$'s (i.e., not necessarily equal to $0$), and we emphasize some
identities.

In Section~\ref{large}, we describe some properties of the solutions of the linear system \eqref{eq}, with the
conditions $l_1,\ldots,l_{K+2}\ge0$, when $K$ is \emph{large}, that is, $K\ge L$. We also prove Corollary \ref{systneg}.

We refer the reader to \cite{Syst}, for instance, for known properties of these solutions.


\subsection{Small-size system: \texorpdfstring{$K\in\{0,\ldots,L+1\}$}
{Kin\{0,\ldots,L+1\}}} \label{small}


The following proposition gives an explicit form for the solution of \eqref{eq}, with $l_{K+2}=0$, when $K$ is
\emph{small}, that is, $K\in\{0,\ldots,L+1\}$. As we prove in Section~\ref{aswei}, these solutions correspond to the
asymptotic normalized local times of the walk, on the corresponding events.

\begin{prop}\label{systpos}
If $\al\in(\al_{L+1},\al_L)$ and $K\in\{0,\ldots,L+1\}$, then the solution of \eqref{eq}, with the extra condition
$l_{K+2}=0$, is unique and satisfies
\begin{equation}
\label{lj}
\qquad l_j=\frac{\sin (((K+2-j)/2)\om )\sin ({j\om}/{2} )}{Z}\qquad\mbox{for all } j\in\{0,
\ldots,K+2\},
\end{equation}
where $Z$ is a normalisation constant. Moreover,
\begin{equation}
\label{d01} d_0=-d_{K+1}=-\al\frac{\sin(({(K+3)}/{2})\om)\sin({\om}/{2})}{Z}.
\end{equation}
The quantities $l_1,\ldots,l_{K+1}$ are positive and if $K\in\{0,\ldots,L-1\}$ (resp., $K\in\{L,L+1\}$) then $d_0<0$ (resp.,
$d_0>0$).

Moreover, \eqref{lj} and \eqref{d01} hold when $\al=\al_L$ and $K\in\{1,\ldots,L-1\}$.
\end{prop}

\begin{rem}\label{notationd0}
We will later deal with two different systems \eqref{eq} and $(E_{K'})$ together, with $K\neq K'$. We will then use the
notations $d_0(K)$, $d_{K+1}(K)$ and $d_0(K')$, $d_{K'+1}(K')$ in order to avoid any confusion.
\end{rem}

\begin{pf*}{Proof of Proposition \ref{systpos}}
All the arguments in the beginning of this proof, until \eqref{lj3}, hold for any $K\in\mathbb{N}$ and will be used
later.

Let us first emphasize the fact that the set of solutions of the system $d_1=\cdots=d_K=0$ is $3$-dimensional. Indeed, let
$(l_0,\ldots,l_{K+2})$ be a real-valued vector satisfying $d_1=\cdots=d_K=0$. If we fix $l_0$, $l_1$ and $l_2$, then it is
easy to see that, by induction, we can find a unique solution. It implies that the set of solutions is
$3$-dimensional.

Moreover, these solutions satisfy linear recurrence relations with characteristic polynomial $P(X)=\al
X^3-X^2+X-\al=\al(X-1)(X-e^{i\om})(X-e^{-i\om})$, hence the general solutions to \eqref{eq} are such that
\begin{equation}
\label{ljgen1} l_j=A+Be^{ij\om}+Ce^{-ij\om}\qquad
\mbox{for all } j\in\{0,\ldots,K+2\},
\end{equation}
for some constants $A, B, C\in\mathbb{C}$. Recall that $\om$ is such that $\cos(\om)=(1-\al)/2\al$.

Then, using the conditions $l_0=0$, $l_{K+2}=0$ and $\sum l_j =1$, we can compute $A$, $B$ and $C$, which will imply
the uniqueness of the solution $(l_0,\ldots,l_{K+2})$ of \eqref{eq} with $l_{K+2}=0$.

First, the condition $l_0=A+B+C=0$ implies
\begin{equation}
\label{ljgen} l_j=A \bigl(1-e^{ij\om} \bigr)-C
\bigl(e^{ij\om}-e^{-ij\om} \bigr).
\end{equation}

Now, let us write $l_j$ as a function of $A$ and $l_{j_0}$, for some $j_0\in\{1,\ldots,(L+3)\wedge(K+2)\}$.

As $2\pi/(L+3)<\om<2\pi/(L+2)$, we have $e^{ij_0\om}-e^{-ij_0\om}\neq0$, for any $j_0\in\{1,\ldots,(L+3)\wedge(K+2)\}$.
Now fix $j_0\in\{1,\ldots,(L+3)\wedge(K+2)\}$, then \eqref{ljgen} applied to $j=j_0$ implies that
\[
C=\frac{A (1-e^{ij_0\om} )-l_{j_0}}{e^{ij_0\om}-e^{-ij_0\om}},
\]
and, for all $j\in\{1,\ldots,K+2\}$, we have
\begin{eqnarray}\label{lj3}
l_j&=& A \bigl(1-e^{ij\om} \bigr)+ \bigl(A
\bigl(e^{ij_0\om}-1 \bigr)+l_{j_0} \bigr)\frac{e^{ij\om}-e^{-ij\om}}{e^{ij_0\om}-e^{-ij_0\om}}
\nonumber
\\
&=&\frac{\sin(j\om)}{\sin(j_0\om)}l_{j_0}+A\frac{\sin(j_0\om)-\sin((j_0-j)\om)-\sin(j\om)}{\sin(j_0\om)}
\nonumber\\[-8pt]\\[-8pt]\nonumber
&=&\frac{\sin(j\om)}{\sin(j_0\om)}l_{j_0}+A\frac{\cos(({j_0}/{2})\om)-\cos((({j_0}/{2})-j)\om)}{\cos(({j_0}/{2})\om)}
\nonumber
\\
&=&\frac{\sin(j\om)}{\sin(j_0\om)}l_{j_0}-2A\frac{\sin((({j_0-j})/{2})\om)\sin(j\om/{2})}{\cos(({j_0}/{2})\om)}.\nonumber
\end{eqnarray}
The previous arguments hold for any $K\in\mathbb{N}$, but we now focus on the case $K\in\{0,\ldots,L+1\}$. We can apply \eqref{lj3} to $j_0=K+2$, and use the conditions $l_{K+2}=0$ and $\sum l_j=1$. This
yields
\begin{eqnarray}
l_j&=&\frac{\sin (((K+2-j)/2)\om )\sin (j\om/2)}{\sum_{i=1}^{K+1}\sin (((K+2-i)/2)\om )\sin ({i\om}/{2} )}
\nonumber\\[-8pt]\\[-8pt]\nonumber
&=&\frac{\cos ((((K+2)/2)-j)\omega )-\cos (((K+2)/2)\omega )}{\sum_{i=1}^{K+1} [\cos ((((K+2)/2)-i)\omega )-
\cos (((K+2)/2)\omega ) ]}.\label{labelattias}
\end{eqnarray}
Notice that, if $\al\in(\al_{L+1},\al_L)$ and $K\in\{0,\ldots,L+1\}$, then for all $j\in\{1,\ldots, K+1\}$:
\[
\sin \biggl(\frac{K+2-j}{2}\om \biggr)>0\quad\mbox{and}\quad\sin \biggl(
\frac{j\om}{2} \biggr)>0.
\]
We have now proved \eqref{lj}.

Let us compute $d_0=l_0-l_1+\al l_2$. As the solution is symmetric, that is, $l_j=l_{K+2-j}$, we have that
$d_{K+1}=-d_0$.

There is an easy way to compute $d_0$. Indeed, let us extend the definition \eqref{lj} of $l_j$ to $j=-1$, then, using
the fact that the system is a part of a bi-infinite Fibonacci-type sequence (for which $d_j=0$ for all
$j\in\mathbb{Z}$), one can show that $-\al l_{-1}+l_0-l_1+\al l_2=0$. Then it implies
\[
d_0=\al l_{-1}=-\al\frac{\sin (((K+2+1)/2)\om )\sin ({\om}/{2} )}{\sum_{i=1}^{K+1}\sin
(((K+2-i)/2)\om )\sin (i\om/2 )},
\]
which completes the proof.
\end{pf*}

\begin{rem}\label{premsyst}
The critical case $\al=\al_L$ is more difficult to describe. When $K=L$, then the conditions $l_0=0$ and $l_{L+2}=0$
are equivalent, hence the solution is not unique anymore, and the set of solutions is in fact $1$-dimensional. More
generally, the solution is not unique as soon as $(K+2)\om$ is a multiple of $2\pi$.
\end{rem}

\begin{rem}\label{remsyst}
If $\al\in(\al_{L+1},\al_L)$ is such that $\omega/\pi$ is irrational, then the solution of the system, with extra condition
$l_{K+2}=0$, is unique for any $K$. Moreover, we can find infinitely many $K$'s such that $l_1,\ldots,l_{K+1}, d_0$ are
positive and $d_{K+1}$ is negative.
Indeed, this occurs when the distance between $(K+2)\omega/2$ and $(2\mathbb{Z}+1)\pi$ is smaller than the distance
between $j\omega/2$ and $(2\mathbb{Z}+1)\pi$, for all $j\in\{1,\ldots,K,K+4\}$. One can convince itself of that fact by using
\eqref{labelattias} and drawing a picture of the unit circle. We refer the reader to \cite{ETW} for more details.

Nevertheless, when $K\ge L$, there can be more than one solution, some $l_j$'s can be nonpositive and $d_0$,
$d_{K+1}$ can take arbitrary signs.
\end{rem}


\subsection{Associated affine medium-size system: $K=L$} \label{medium}


The main purpose of this section is to prove the following proposition and its corollary, which will be used to prove
Proposition \ref{proprang}.

Let us first define the affine version of the system $(E_L)$.

Fix $L\ge1$ and assume that $\al\in(\al_{L+1},\al_L)$, where $(\al_L)_{L\ge1}$ is defined in \eqref{defalpha}. Fix also
$L$ real numbers $d_j$, for $j\in\{1,\ldots,L\}$. Then consider the following system on $(l_0,\ldots,l_{L+2})$:
\[
AS(d_1,\ldots,d_L) \cases{
\displaystyle l_0=l_{L+2}=0,
\vspace*{3pt}\cr
\displaystyle d_j=-\alpha l_{j-1}+l_j-l_{j+1}+ \alpha l_{j+2},\qquad\mbox{for all }j\in\{1,\ldots,L\},
\vspace*{3pt}\cr
\displaystyle \sum_{j=1}^{L+1} l_j=1.}
\]

Define also the quantities
\begin{eqnarray*}
d_0&:=&-l_1+\al l_2,
\\
d_{L+1}&:=&-\al l_L+l_{L+1}.
\end{eqnarray*}

\begin{prop}\label{gensyst}
Assume that $\al\in(\al_{L+1},\al_L)$. There exists a unique solution $(l_0,\ldots,l_{L+2})$ of $AS(d_1,\ldots,d_L)$, and
there exist positive constants $c_1,\ldots,c_L$, that do not depend on $d_1,\ldots,d_L$, such that
\begin{eqnarray*}
d_{L+1}&=&-d_0(L)-\sum_{k=1}^L
c_kd_k,
\\
d_0&=&d_0(L)-\sum_{k=1}^L
c_{L+1-k}d_k,
\end{eqnarray*}
where $d_0(L)$ is a positive constant defined in Remark \ref{notationd0}.
\end{prop}

\begin{coro}\label{corgensyst}
Assume that $\al\in(\al_{L+1},\al_L)$, and that $(l_0,\ldots,l_{L+2})$
is a solution of the system
\begin{eqnarray*}
l_0&=&0,\qquad l_{L+2}\ge0,
\nonumber
\\
d_j&=&-\alpha l_{j-1}+l_j-l_{j+1}+
\alpha l_{j+2}\qquad\mbox{for all }j\in\{1,\ldots,L\},
\\
\sum_{j=1}^{L+1} l_j&=&1.
\end{eqnarray*}
Then the following inequality holds:
\begin{eqnarray*}
d_0&\ge&d_0(L)-\sum_{k=1}^L
c_{L+1-k}d_k.
\end{eqnarray*}
\end{coro}

\begin{pf} 
Define $\tilde{l}_j=l_j$ for $j\in\{0,\ldots,L+1\}$, and $\tilde{l}_{L+2}=0$. Then $\widetilde{d}_0=d_0$, and
notice that $(\tilde{l}_0,\ldots,\tilde{l}_{L+2})$ is a solution of $AS(d_1,\ldots, d_{L-1},d_L-\al l_{L+2})$, hence
\[
d_0=d_0(L)-\sum_{k=1}^L
c_{L+1-k}d_k+\al c_1 l_{L+2},
\]
which enables us to conclude, using that $l_{L+2}\ge0$.
\end{pf}

\begin{pf*}{Proof of Proposition \ref{gensyst}}
First, notice that the matrix associated to the system $AS(d_1,\ldots,d_L)$ can be written as
\[
M:=\pmatrix{ 1 & 0 & \cdots & \cdots & \cdots & \cdots & 0
\cr
0 & \cdots &
\cdots & \cdots & \cdots &0 & 1
\cr
0& 1 & \cdots & \cdots & \cdots & 1 & 0
\cr
-
\al&1&-1&\al&0& \cdots &0
\cr
&\ddots&\ddots&\ddots&\ddots&&
\cr
&&\ddots&\ddots&
\ddots&\ddots&
\cr
0& \cdots &0&-\al&1&-1&\al},
\]
where the size of $M$ is $(L+3)^2$, and $AS(d_1,\ldots,d_L)$ is then rewritten as
\[
M\pmatrix{ l_0
\cr
\vdots
\cr
l_{L+2}}= \pmatrix{ 0
\cr
0
\cr
1
\cr
d_1
\cr
\vdots
\cr
d_L}.
\]

By Proposition \ref{systpos}, there exists a unique solution to $(E_L)$ with $l_{L+2}=0$, which is equivalent to
$AS(0,\ldots,0)$. Therefore, $M$ has a nonzero determinant, thus it is invertible.

Now, we first want to prove that, for some constants $c_1$, \ldots, $c_L$,
\begin{eqnarray}
d_{L+1}&=&-d_0(L)-\sum_{k=1}^L
c_kd_k,\label{dl1gene}
\end{eqnarray}
which, by symmetry, will imply
\begin{eqnarray*}
d_0&=&d_0(L)-\sum_{k=1}^L
c_{L+1-k}d_k,
\end{eqnarray*}
and we will then prove that these constants are positive. Notice that
\begin{eqnarray*}
d_{L+1}&=&-\al l_L+l_{L+1}=\pmatrix{ 0& \cdots
&0&-\al&1&0}M^{-1}\pmatrix{ 0
\cr
0
\cr
1
\cr
d_1
\cr
\vdots
\cr
d_L}
\\
&=&\pmatrix{ 0& \cdots &0&-\al&1&0}M^{-1} \Biggl(v_3 +
\sum_{k=1}^L d_k
v_{3+k} \Biggr),
\end{eqnarray*}
where $(v_k)_k$ is the canonical basis of $\mathbb{R}^{L+3}$. Then, using Proposition \ref{systpos} and the notations
of Remark \ref{notationd0}, the vector $M^{-1}v_3$ is the solution to $(E_L)$ with $l_{L+2}=0$, which yields
\[
d_0(L)={-}\pmatrix{0& \cdots &0&-\al&1&0} M^{-1}v_3,
\]
and we prove \eqref{dl1gene} by defining, for all $k\in\{1,\ldots,L\}$,
\[
c_k={-}\pmatrix{ 0& \cdots &0&-\al&1&0}M^{-1}v_{3+k}.
\]

It remains to prove that, for all $k\in\{1,\ldots,L\}$, $c_k>0$. Recall that these constants do not depend on
$d_1,\ldots,d_L$. We proceed by induction on $k=1,\ldots,L$.

First, for $k=1$, choose $d_2=\cdots=d_L=0$, and\vspace*{1pt} instead of fixing $d_1$ let us fix $l_1=0$. Then the shifted vector
$(\tilde{l}_0,\ldots,\tilde{l}_{(L-1)+2}):=(l_1,\ldots,l_{L+2})$ is the unique solution of $(E_{L-1})$ with
$\tilde{l}_{(L-1)+2}=0$. Therefore, using Proposition \ref{systpos} and the notation of Remark \ref{notationd0}, the
solution is unique, $d_1=d_0(L-1)<0$ and $d_{L+1}=-d_0(L-1)>0$. On the other hand, \eqref{dl1gene} implies that
$c_1= (-d_0(L)-d_{L+1} )/ d_1>0$, which finishes the case $k=1$.

Now, assume that, for some $k\in\{2,\ldots,L\}$, $c_1,\ldots,c_{k-1}>0$, and let us prove that $c_{k}>0$. As previously, fix
$d_{k+1}=\cdots=d_L=0$ (except if $k=L$), and fix $l_1=\cdots=l_k=0$. Then the shifted vector
$(\tilde{l}_0,\ldots,\tilde{l}_{(L-k)+2}):=(l_{k},\ldots,l_{L+2})$ is the unique solution of $(E_{L-k})$ with
$\tilde{l}_{(L-k)+2}=0$. Thus, $d_1=\cdots=d_{k-2}=0$ (if $k\ge3$), $d_{k-1}=\al l_{k+1}>0$, $d_k=d_0(L-k)<0$ and
$d_{L+1}=-d_0(L-k)>0$. On the other hand, \eqref{dl1gene} implies that
$c_k= (-d_0(L)-d_{L+1}-c_{k-1}d_{k-1} )/d_k>0$, which completes the proof.
\end{pf*}


\subsection{Nonnegative solutions of the large-size system, 
that is, \texorpdfstring{$K\ge L$}{K>= L}, and proof of 
Corollary \texorpdfstring{\protect\ref{systneg}}{2.3}}
\label{large}


Recall the definition of the system \eqref{eq} introduced in Section~\ref{trigo}. When $K\ge L$, as it is noticed
in Remarks \ref{premsyst} and \ref{remsyst}, the set of solutions can be $1$-dimensional, and {a priori} we cannot
determine the signs of $d_0$ nor $d_{K+1}$.

But, the following result, which implies Corollary \ref{systneg}, is helpful in order to understand the behavior of the
solutions, even without uniqueness.

\begin{prop}\label{system}
\textup{(i)} Assume that $\al\in(\al_{L+1},\al_L)$ and $K\geq L$. Then, if some real-valued vector $(l_0,\ldots,l_{K+2})$ satisfies
\eqref{eq} and if $l_1,\ldots,l_{K+2}\geq0$, then $d_0\ge c(K)>0$, where $c(K)$ is a positive constant depending only on
$\al$ and $K$.

It implies that if $l_{K+2}=0$, then, by symmetry, we have $d_{K+1}\leq-c(K)$.

\textup{(ii)} Assume that $\al\in(\al_{L+1},\al_L]$. If $K\ge L$ and if $(l_0,\ldots,l_{K+2})$ satisfies \eqref{eq}, then
\begin{eqnarray}
\label{y}
\qquad && l_{L+2}-\al l_{L+1}
\nonumber\\[-8pt]\\[-8pt]\nonumber
&&\qquad =-\frac{\sin(((L+2)/2)\om)}{\sin((L/2)\om)}l_1+2
\al\frac{\cos(\om/2)\sin(((L+3)/2)\om)}{\sin((L/2)\om)}l_{L+1}.
\end{eqnarray}
Moreover, if $\al\in(\al_{L+1},\al_L)$ and if $l_1,\ldots,l_{K+2}\geq0$, then $l_{L+2}-\al l_{L+1}<-c(K)$, where we can choose $c(K)$ such that it is
the same positive constant as in \textup{(i)} by convenience.
\end{prop}

\begin{pf}
Let us prove the first point:

\begin{longlist}[(iii)]
\item[(i)] Recall that $\al\in(\al_{L+1},\al_L)$. Let us first prove that $\sum_{k=1}^{L+1} l_k> c$, for some positive
constant $c$.

Recall that $l_0=0$ and $d_1=\cdots=d_K=0$. If $l_1<c$ and $l_2<c$, we can prove, by induction, that
\begin{eqnarray}
\label{inductionpds} l_j\leq 6^{j-1}c\qquad\mbox{for all } j\in
\{1,\ldots,K+2\},
\end{eqnarray}
which implies, as soon as $c$ is small enough, that $\sum_{j=1}^{K+1}l_j<1$ which contradicts the weight condition. In
order to prove \eqref{inductionpds} by induction, recall that $\al>1/3$ and notice that, for all $j\in\{2,\ldots,K+1\}$,
$d_{j-1}=-\al l_{j-2}+l_{j-1}-l_j+\al l_{j+1}=0$ and $l_{j-1}\ge0$, which implies that
\[
l_{j+1}\leq l_{j-2}+l_j/\al.
\]

Now, define, for all $j\in\{0,\ldots,L+2\}$,
\[
\tilde{l}_j:=\frac{l_j}{\sum_{k=1}^{L+1} l_k}.
\]
Now, we use Corollary \ref{corgensyst} with $(\tilde{l}_0,\ldots,
\tilde{l}_{L+2})$, which enables us to conclude
\[
d_0\ge d_0(L)\times\sum_{k=1}^{L+1}
l_k> d_0(L)\times c>0.
\]

\item[(ii)] Here, we assume $\al\in(\al_{L+1},\al_L]$.
To prove (ii), we consider the system \eqref{eq}, together with the conditions $l_1,\ldots,l_{K+2}\geq0$ and let $l_{L+1}$
be a parameter in order to show the $l_j$'s in terms of $l_{L+1}$.

We can use the arguments at the beginning of the proof of Proposition \ref{systpos}. In particular, noticing that
$\sin((L+1)\om)\neq0$, we can apply \eqref{lj3} to $j_0=L+1$, which yields, for all $j\in\{1,\ldots,K+2\}$, and for some
constant $A$,
\begin{eqnarray}
l_j&=&\frac{\sin(j\om)}{\sin((L+1)\om)}l_{L+1}-2A\frac{\sin(((L+1-j)/2)\om)\sin(j\om/2)}{\cos(((L+1)/2)\om)}.\label{lj31}
\end{eqnarray}
Then we can show $A$ in terms of $l_{L+1}$ and $l_1$:
\begin{eqnarray}
\label{lj32} 2A\frac{\sin((L/2)\om)\sin(\om/2)}{\cos(((L+1)/2)\om)}=\frac{\sin(\om)}{\sin((L+1)\om)}l_{L+1}-l_1,
\end{eqnarray}
 and compute
\begin{eqnarray*}
l_{L+2}-\al l_{L+1}&=& \biggl(\frac{\sin((L+2)\om)}{\sin((L+1)\om)}-\al
\biggr)l_{L+1}+2A\frac{\sin(((L+2)/2)\om)\sin(\om/2)}{\cos(((L+1)/2)\om)}
\\
&=& \biggl(\frac{\sin((L+2)\om)}{\sin((L+1)\om)}-\al+\frac{\sin(\om)\sin(((L+2)/2)\om)}{\sin((L+1)\om)\sin(L\om/2)} \biggr)l_{L+1}
\\
&&{}-\frac{\sin(((L+2)/2)\om)}{\sin(L\om/2)}l_1.
\end{eqnarray*}
Now, in order to prove \eqref{y}, define $E$ such that
\begin{eqnarray*}
&& \frac{\sin((L+2)\om)}{\sin((L+1)\om)}-\al+\frac{\sin(\om)\sin(((L+2)/2)\om)}{\sin((L+1)\om)\sin(L\om/2)}
\\
&&\qquad =-\frac{\al
E}{4\sin(L\om/2)\sin((L+1)\om)},
\end{eqnarray*}
then, using that $\al^{-1}=1+2\cos(\om)$ and the exponential forms of $\sin$ and $\cos$, we have
\begin{eqnarray*}
E&=& \bigl(e^{i(L/2)\om}-e^{-i(L/2)\om} \bigr) \bigl(e^{i(L+2)\om}-e^{-i(L+2)\om}
\bigr) \bigl(1+e^{i\om}+e^{-i\om} \bigr)
\\
&&{}- \bigl(e^{i(L/2)\om}-e^{-i(L/2)\om} \bigr) \bigl(e^{i(L+1)\om}-e^{-i(L+1)\om}
\bigr)
\\
&&{} + \bigl(e^{i((L+2)/2)\om}-e^{-i((L+2)/2)\om} \bigr) \bigl(e^{i\om}-e^{-i\om}
\bigr) \bigl(1+e^{i\om}+e^{-i\om} \bigr)
\\
&=&e^{(3/2)L\om} \bigl(e^{i2\om}+e^{i3\om}
\bigr)+e^{-(3/2)L\om} \bigl(e^{-i2\om}+e^{-i3\om} \bigr)
\\
&&{}-e^{(L/2)\om} \bigl(1+e^{-i\om} \bigr)-e^{-(L/2)\om}
\bigl(1+e^{i\om} \bigr)
\\
&=&4\cos \biggl(\frac{\om}{2} \biggr) \biggl(\cos \biggl( \biggl(
\frac{3}{2}L+\frac{5}{2} \biggr)\om \biggr)-\cos \biggl(
\frac{L-1}{2}\om \biggr) \biggr)
\\
&=&-8\cos \biggl(\frac{\om}{2} \biggr)\sin \biggl(\frac{L+3}{2}\om
\biggr)\sin \bigl((L+1)\om \bigr),
\end{eqnarray*}
which proves \eqref{y}.\vadjust{\goodbreak}

Let us prove the second part of (ii). Assume that $\al\in(\al_{L+1},\al_L)$, and recall that it is equivalent to
$2\pi/(L+3)<\om<2\pi/(L+2)$. Here, we use the notation $\verb?Cst?(\al,K)$ to denote a positive constant, depending on
$\al$ and $K$, which can be different from line to line or between two (in-)equalities.

Let us show that there exists a constant $\tilde{c}$ such that $l_1+l_{L+1}\ge \tilde{c}$.
It will imply the conclusion, using \eqref{y}, that is, $l_{L+2}-\al l_{L+1}\leq -\verb?Cst?(\al,K)(l_1+l_{L+1})\le
-\verb?Cst?(\al,K)$.

By contradiction, assume that $l_1+l_{L+1}< \tilde{c}$, then we obviously have $l_1<\tilde{c}$ and $l_{L+1}<\tilde{c}$.
Formula \eqref{lj32} implies that
\[
0\leq 2A<\verb?Cst?(\al,K)\times \tilde{c},
\]
then \eqref{lj31} implies that
\[
\sum_{j=1}^{K+1}l_j<(K+1)
\verb?Cst?(\al,K)\times \tilde{c}.
\]
Now $\tilde{c}$ small enough would imply $\sum_{j=1}^{K+1}l_j<1$, and thus contradicts the weight condition.\quad\qed
\end{longlist}\noqed
\end{pf}


\section{The walk has finite range: Proof of Proposition 
\texorpdfstring{\protect\ref{proprang}}{2.1}} \label{finiterange}


In the section, we prove Proposition \ref{proprang}, which states that the walk has finite range.
The general strategy of the proof is inspired by the one of \cite{PV}. This proof is
quite simple, making use of Corollary \ref{corgensyst}.

\begin{pf*}{Proof of Proposition \protect\ref{proprang}}
Let us prove that there exists a constant $p>0$ such that, for all $x\in\mathbb{Z}$ such that $x<-L-1$, if the walker
visits $x+L+1$, then, with probability at least $p$, the walker does not visit the site $x-1$, that is,
\[
\Pb (x-1\in R|x+L+1\in R )<1-p.
\]
Note that $p$ will not depend on $x$.

This will imply that
\[
\Pb (\inf R=-\infty )=0.
\]
Then the symmetric argument will hold and imply the conclusion.

For all $y\in\mathbb{Z}$ and $m\in\mathbb{N}$, let $u_m(y)$ be the first time $k$ such that $Z_k(y)=m$, that is,
\begin{eqnarray*}
u_m(y):=\inf \bigl\{k\in\mathbb{N}\dvtx Z_k(y)=m \bigr
\}.
\end{eqnarray*}
If
$Z_\infty(y)<m$, then $u_m(y)=+\infty$.

Now fix $x\in\mathbb{Z}$ such that $x<-L-1$. For all $m\in\mathbb{N}$, we have
\[
\Pb (X_{k+1}=x+L+1 |\Fk )\ge\frac{\1_{ \{u_m(x+L)=k,Z_k(x+L-1)=0 \}}}{1+e^{2(2m-1)\beta}}.
\]
Together with
\begin{eqnarray*}
&& \bigl\{Z_{u_{m}(x+L)}(x+L-1)=0 \bigr\}
\\
&&\qquad =\bigcap_{k=1}^{m-1} \bigl
\{X_{u_k(x+L)+1}=x+L+1, u_k(x+L)<\infty \bigr\}\cup \bigl
\{u_k(x+L)=+\infty \bigr\},
\end{eqnarray*}
this implies that, for all $m\in\mathbb{N}$,
\begin{eqnarray*}
\Pb \bigl(Z_{u_{m}(x+L)}(x+L-1)=0|\mathcal{F}_{u_1(x+L)} \bigr) &\ge&
\1_{ \{u_1(x+L)<\infty \}}\eps_1,
\end{eqnarray*}
where $\eps_1:=\eps_1(m)$ is a constant depending on $m$, but not on $x$.

From now on, we work on the event $ \{Z_{u_{m}(x+L)}(x+L-1)=0 \}$. Notice that, on this event, for any
$n\in\mathbb{N}$, $u_n(x)\ge u_1(x)>u_m(x+L)$.

Define, for all $k\in\mathbb{N}$, the time spent in the interval $[x,x+L+1]$, that is,
\begin{eqnarray}
l_k^{(x)}:=l_k(x)+l_k(x+1)+
\cdots+l_k(x+L+1),\label{deflx111}
\end{eqnarray}
and notice that $l_{u_m(x+L)}^{(x)}=2m-1$.\vspace*{1.5pt}

Fix $\eps_2>0$, and define, for all $n\in\mathbb{N}$, the event $U_n$ which is such that the walker performs an upstream
jump to the left on $\{x+1,\ldots,x+L\}$ of intensity at least $\eps_2$ between the times $u_m(x+L)$ and $u_n(x)$, that is,
\begin{eqnarray*}
U_n&:=&\biggl\{\exists k\in \bigl[u_m(x+L),u_n(x)
\bigr], \exists j\in\{1,\ldots,L\}\dvtx \frac{\Delta_k(x+j)}{l_k^{(x)}}>\eps_2,
\\
&&\hspace*{117pt} X_k=x+j\mbox{ and }X_{k+1}=x+j-1 \biggr\}.
\end{eqnarray*}
If we define $u_\infty(x)=\infty$, then $\bigcup_nU_n=U_\infty$.

Now, notice that, on the event $ \{Z_{u_{m}(x+L)}(x+L-1)=0 \}$, we have
\begin{eqnarray}\label{incluuinf}
&& \bigl\{Z_\infty(x-1)>0 \bigr\}\nonumber
\\
&&\qquad \subset \bigcup
_{n=1}^\infty \bigl\{u_1(x-1)=
u_n(x)+1 \bigr\}\cap \bigl\{u_n(x)<\infty \bigr\}
\\
&&\qquad \subset \Biggl(\bigcup_{n=1}^\infty \bigl
\{u_1(x-1)=u_n(x)+1 \bigr\}\cap \bigl\{u_n(x)<
\infty \bigr\}\cap U_n^c \Biggr)\cup U_\infty.\nonumber
\end{eqnarray}

It is easy to upper-bound the probability of $U_\infty$ by writing
\[
\Pb (U_\infty |\mathcal{F}_{u_m(x+L)} )\le\sum
_{k=l_{u_m(x+L)}^{(x)}}^{+\infty} e^{-2\beta\eps_2 k}.
\]
Then let us focus on the other union of events in \eqref{incluuinf}.
Until the conclusion of the proof, fix $n\in\mathbb{N}$, assume that $u_n(x)<+\infty$ and $l_{u_n(x)}(x)=0$, which is
necessary to have $u_1(x-1)=u_n(x)+1$. Moreover, assume that $U_n^c$ holds and let us show that, at time $u_n(x)$, the
stream at $x$ is strongly positive, thus the probability for the walker to jump on $x-1$ is small.

For $j\in\{1,\ldots,L\}$, let $k_j$ denote the last time before $u_n(x)$ such that $X_{k_j}=x+j$, that is,
\begin{eqnarray*}
k_j:=\max \bigl\{k\le u_n(x)\dvtx X_k=x+j
\bigr\}.
\end{eqnarray*}
Then, as $X_{k_j+1}=x+j-1$ and
$U_n^c$ holds, we have $\Delta_{k_j}(x+j)/l_{k_j}^{(x)}\le\eps_2$. Therefore, we have that, for any $j\in\{1,\ldots,L\}$,
\begin{eqnarray*}
\frac{\Delta_{u_n(x)}(x+j)}{l_{u_n(x)}^{(x)}}&=&\frac{-\alpha
 (l_{u_n(x)}(x+j-1)-l_{k_j}(x+j-1) )+\Delta_{k_j}(x+j)+1}{l_{u_n(x)}^{(x)}}
\\
&\le&2\eps_2,
\end{eqnarray*}
where the inequality holds as soon as we choose $m>(1+\eps_2)/(2\eps_2)$.

Moreover, recall that $l_{u_n(x)}(x)=0$ and notice, using \eqref{deflx111}, that
\[
\sum_{j=1}^{L+1} \frac{l_{u_n(x)}(x+j)}{l_{u_n(x)}^{(x)}}=1.
\]
Therefore, using Corollary \ref{corgensyst}, with $l_j:=l_{u_n(x)}(x+j)/l_{u_n(x)}^{(x)}$ for all $j\in\{0,\ldots,L+2\}$, we
obtain
\begin{eqnarray*}
\frac{\Delta_{u_n(x)}(x)}{l_{u_n(x)}^{(x)}}&\ge&d_0(L)-\sum_{k=1}^L
\frac{\Delta_{u_n(x)}(x+k)}{l_{u_n(x)}^{(x)}}c_{L+1-k}
\\
&\ge&d_0(L)-2\eps_2\sum_{k=1}^Lc_k>
\frac{d_0(L)}{2}>0,
\end{eqnarray*}
as soon as $\eps_2$ is small enough, depending only on the constants $c_k$'s and $d_0(L)$, where the $c_k$'s are
positive constants defined in Proposition \ref{gensyst} and $d_0(L)$ is defined in Remark \ref{notationd0}.

Using \eqref{incluuinf} and using that $l_{u_m(x+L)}^{(x)}=2m-1$, we are finally able to conclude, as soon as $m$ is
large enough, as we have
\begin{eqnarray*}
&& \Pb \bigl(Z_{\infty}(x-1)>0 |\mathcal{F}_{u_m(x+L)} \bigr)
\1_{ \{Z_{u_m(x+L)}(x+L-1)=0 \}}
\\
&&\qquad \le \sum_{k=l_{u_m(x+L)}^{(x)}}^{+\infty}
e^{-2\beta(d_0(L)/2)k}+\sum_{k=l_{u_m(x+L)}^{(x)}}^{+\infty}
e^{-2\beta\eps_2 k}.
\end{eqnarray*}\upqed
\end{pf*}


\section{Asymptotic streams} \label{aswei}


Given $L\ge1$, assume that $\al\in(\al_{L+1},\al_L)$ and fix $x\in\mathbb{Z}$ and $K\in\mathbb{N}$. The goal is here to
prove Proposition \ref{lemdeltak1} in Section~\ref{aswei1}, and Lemma \ref{somfinie} in Section~\ref{aswei2}.


\subsection{Proof of Proposition \texorpdfstring{\protect\ref{lemdeltak1}}{2.2}} \label{aswei1}


Recall that Proposition \ref{lemdeltak1} states that, almost surely, we have
\begin{eqnarray}
\label{butprop} \bigl\{R'=\{x,\ldots,x+K+1\} \bigr\} \subset \bigcap
_{j=x+1}^{x+K} \biggl\{\lim
_{k\rightarrow+\infty} \frac{\Delta_k(j)}{k}=0 \biggr\}.
\end{eqnarray}

Proposition \ref{lemdeltak1} extends to general $K\in\mathbb{N}$ the result of Erschler, T\'oth and Werner in
\cite{ETW}, initially proved for $K<L+1$. Its proof is mainly combinatorial. Let us first state and prove two lemmas that will be used several times to prove Proposition~\ref{lemdeltak1}.

The only probabilistic tool of the proof is Borel--Cantelli lemma, which is used to prove the following result.

\begin{lem}\label{stuck} Let $\eps>0$. There exist two random times $N_1$ and $N_2$ depending on $\eps$ such that:

\begin{longlist}[(ii)]
\item[(i)] For all $j\in\mathbb{Z}$, and for all $k\ge N_1$, if $\Delta_k(j)/k\ge\eps$ and $X_k=j$, then $X_{k+1}=j+1$, and
if $\Delta_k(j)/k\le-\eps$ and $X_k=j$, then $X_{k+1}=j-1$.

\item[(ii)] For all $k\ge N_2$, for all $j\in\mathbb{Z}$, if $X_k<j$, then $\Delta_k(j)/k<3\eps/2$ and if \mbox{$X_k>j$}, then
$\Delta_k(j)/k>-3\eps/2$.
\end{longlist}
\end{lem}

\begin{pf}
Let us prove the first point.

Recall that, for all $j\in\mathbb{Z}$, we have
\[
\Pb (X_{k+1}=X_k+1 | \Fk )=\frac{1}{1+e^{-2\beta\Delta_k(X_k)}},
\]
which yields
\[
\Pb \biggl(X_k=j,X_{k+1}=j+1,\frac{\Delta_k(j)}{k}\le -\eps \Big|
\Fk \biggr)\le\frac{\mathds{1}_{\{X_k=j\}}}{1+e^{2\beta\eps k}}.
\]
Therefore, we have that
\begin{eqnarray*}
\Pb \biggl(\bigcup_{j\in\mathbb{Z}} \biggl
\{X_k=j,X_{k+1}=j\pm1,\frac{\mp\Delta_k(j)}{k}\ge \eps \biggr\} \Big|
\Fk \biggr)\le\frac{2}{1+e^{2\beta\eps k}},
\end{eqnarray*}
which is summable. Then, by Borel--Cantelli lemma, there exists a random $N_1\in\mathbb{N}$ such that for all
$j\in\mathbb{Z}$, and for all $k\ge N_1$, if $\Delta_k(j)/k\ge\eps$ and $X_k=j$, then $X_{k+1}=j+1$, and if
$\Delta_k(j)/k\le-\eps$ and $X_k=j$, then $X_{k+1}=j-1$.

Let us prove the second point of the lemma. Let $N_2$ be the smallest integer such that $N_2>2(1+\al) N_1/3\eps>N_1$
and $N_2>4/\eps$.

Assume that $k\ge N_2$ and $X_k<j$. Recall that, by definition of $N_1$, the walker cannot jump from $j$ to $j-1$ at
some time $k\ge N_1$ if $\Delta_k(j)/k\ge\eps$. Recall also the definition of the local stream at site $j$ and at time
$k$:
\[
\Delta_k(j)=-\alpha l_k(j-1)+l_k(j)-l_k(j+1)+
\alpha l_k(j+2).
\]

First, if the walker has not visited $j$ between $N_1$ and $k$, then $\Delta_{k}(j)\le\Delta_{N_1}(j)\le
N_1\times(1+\al)$, hence
\[
\frac{\Delta_{k}(j)}{k}\le\frac{\Delta_{N_1}(j)}{N_1}\times\frac{N_1}{N_2}<
\frac{3\eps}{2}.
\]
Otherwise, let $k'$ be the last time before $k$ such that $X_{k'}=j$. Then $X_{k'+1}=j-1$, and therefore
$\Delta_{k'}(j)/k'<\eps$. Note that $\Delta_k(j)\le \Delta_{k'}(j)+1$, hence
\[
\frac{\Delta_{k}(j)}{k}<\frac{3\eps}{2},
\]
which proves the first part of (ii) and the second part is proved by the symmetric argument.
\end{pf}

Recall that $L\ge1$, $\al\in(\al_{L+1},\al_L)$, $K\in\mathbb{N}$ are fixed. The following lemma uses the results of
Section~\ref{trigogo} in order to describe the evolutions of the local times at two sites when the walk is stuck between them.

\begin{lem} \label{stuck2}
Fix $i_1,i_2\in\mathbb{Z}$ such that $0<i_2-i_1<K$. Fix two times $k$ and $k'$ such that $k<k'$.

Assume that the walk remains in $\{i_1,\ldots,i_2+1\}$ between the time $k$ and $k'$ and define
\[
\delta_{\max}:=\max \biggl\{\biggl\llvert \frac{\Delta_n(i)}{n}\biggr
\rrvert\dvtx n\in\bigl[k,k'\bigr],i\in\{i_1+1,
\ldots,i_2\} \biggr\}.
\]
Then there exist two positive constants $d$ and $C_1$, depending on $\al$ and $K$, such that
\begin{eqnarray*}
\frac{\Delta_{k'}(i_1+m(i_2-i_1))}{k'}&>&\frac{\Delta_{k}(i_1+m(i_2-i_1))}{k}\times\frac{k}{k'}
\\
&&{}+d\times\frac{k'-k}{k'}-2C_1\delta_{\max},
\\
\frac{\Delta_{k'}(i_2+1-m(i_2-i_1))}{k'}&<&\frac{\Delta_{k}(i_2+1-m(i_2-i_1))}{k}\times\frac{k}{k'}
\\
&&{}-d\times\frac{k'-k}{k'}+2C_1\delta_{\max},
\end{eqnarray*}
where $m(k):= (k+1)\1_{\{k< L\}}$, for all $k\in\mathbb{N}$.
\end{lem}

Before proving this result, let us emphasize two simple consequences. First, notice that
\begin{eqnarray*}
\frac{\Delta_{k'}(i_1+m(i_2-i_1))}{k'}&>&\min \biggl(\frac{\Delta_{k}(i_1+m(i_2-i_1))}{k},d \biggr)-2C_1
\delta_{\max},
\\
\frac{\Delta_{k'}(i_2+1-m(i_2-i_1))}{k'}&<&\max \biggl(\frac{\Delta_{k}(i_2+1-m(i_2-i_1))}{k},-d \biggr)
+2C_1\delta_{\max}.
\end{eqnarray*}
Second, if
\[
\frac{k'-k}{k}\ge \frac{2C_1\delta_{\max}}{d},
\]
then
\begin{eqnarray}
\frac{\Delta_{k'}(i_1+m(i_2-i_1))}{k'}&>&\frac{\Delta_{k}(i_1+m(i_2-i_1))}{k}\times\frac{k}{k'},\label{LB}
\\
\frac{\Delta_{k'}(i_2+1-m(i_2-i_1))}{k'}&<&\frac{\Delta_{k}(i_2+1-m(i_2-i_1))}{k}\times\frac{k}{k'}.\label{UB}
\end{eqnarray}

\begin{pf*}{Proof of Lemma \ref{stuck2}}
Define
\begin{eqnarray*}
\widetilde{l_i}&=&\frac{l_{k'}(i_1+i)-l_k(i_1+i)}{k'-k}\qquad\mbox{for all } i\in\{0,
\ldots,i_2-i_1+2\}
\end{eqnarray*}
and
\begin{eqnarray*}
\widetilde{d_i}=-\al \tilde{l}_{i-1}+
\tilde{l}_i-\tilde{l}_{i+1}+\al \tilde{l}_{i+2}
\qquad\mbox{for all } i\in\{1,\ldots,i_2-i_1\}.
\end{eqnarray*}
Thus,
\[
\tilde{l}_0=\tilde{l}_{i_2-i_1+2}=0,\qquad \sum
_{i=1}^{i_2-i_1+1}\tilde{l}_i=1
\]
and
\[
\quad \llvert \widetilde{d}_i\rrvert <2\delta_{\max}\times
\frac{k'}{k'-k}\qquad\forall i \in\{1,\ldots,i_2-i_1\}.
\]

Define also $\widetilde{d}_0=-\tilde{l}_1+\al\tilde{l}_{2}$ and
$\widetilde{d}_{i_2-i_1+1}=-\al\tilde{l}_{i_2}+\tilde{l}_{i_2+1}$.

Recall the notation $d_0(K')$ and $d_{K'+1}(K')$ defined in Remark \ref{notationd0}, $c(K')$ defined in
Proposition \ref{system} and  define
\begin{eqnarray}\label{defd}
d &:=& \min \bigl\{ \bigl(d_{K'+1} \bigl(K'
\bigr)=-d_0 \bigl(K' \bigr), K'=1,\ldots,L-1 \bigr),
\nonumber\\[-8pt]\\[-8pt]\nonumber
&&\hspace*{150pt} c(L),\ldots,c(K) \bigr\}>0.
\end{eqnarray}

Let us consider $(\tilde{l}_i)$ as a perturbed solution of the linear system $(E_{i_2-i_1})$ defined and described in Section~\ref{trigogo}. Then we have to distinguish two cases:
$i_2-i_1\le L-1$ and $i_2-i_1\ge L$.

Indeed, if $i_2-i_1\le L-1$, using Proposition \ref{systpos} and using that $(E_{i_2-i_1})$ is a linear system, there exists a positive constant $C_1$ such that
\begin{eqnarray*}
\widetilde{d}_{0}&<&-d+C_1\times2\delta_{\max}
\times \frac{k'}{k'-k},
\\
\widetilde{d}_{i_2-i_1+1}&>&d-C_1\times2\delta_{\max}
\times \frac{k'}{k'-k}.
\end{eqnarray*}

Otherwise, if $i_2-i_1\ge L$, using Proposition \ref{system}, we have similarly
\begin{eqnarray*}
\widetilde{d}_{0}&>&d-C_1\times2\delta_{\max}
\times \frac{k'}{k'-k},
\\
\widetilde{d}_{i_2-i_1+1}&<&-d+C_1\times2\delta_{\max}
\times \frac{k'}{k'-k}.
\end{eqnarray*}
We can conclude by multiplying these quantities by $(k'-k)/k'$ and noticing that, for $i\in\{0,i_2-i_1+1\}$,
\[
\widetilde{d}_i:=\frac{\Delta_{k'}(i_1+i)-\Delta_{k}(i_1+i)}{k'-k}.
\]\upqed
\end{pf*}

Let us now prove Proposition \ref{lemdeltak1},  that is, \eqref{butprop}.

\begin{pf*}{Proof of Proposition \ref{lemdeltak1}}
We assume throughout the proof that the probability event $ \{R'=\{x,\ldots,x+K+1\} \}$ holds.

Let us show that there exists a constant $M:=M(\al,d,C_1)$ such that, for all $j\in\{1,\ldots,K\}$, for all $\eps>0$
small enough, we have
\begin{equation}
\label{deltkgrd} \frac{\Delta_k(x+j)}{k}>-3\eps M^{j(j-1)/2}
\end{equation}
as soon as $k$ is large enough. Then the symmetric argument will hold and enable us to conclude. Note that, in fact,
the constant $M$ depends only on $\al$ and $K$, since $d$ and $C_1$ depend only on $\al$ and $K$.

For reasons that will become clear later, we choose $M$ and $\eps$ such that
\begin{equation}
\label{defMeps} M>\frac{3\times(d+8C_1)}{d}\times(3+\al)\quad\mbox{and}\quad \eps<
\frac{2(1+\al)\wedge d}{3M^{K(K+1)}}.
\end{equation}

As $R'=\{x,\ldots,x+K+1\}$, there exists a random $N_3\in\mathbb{N}$ such that $Z_\infty (i)=Z_{N_3}(i)$ if $i\notin
R'$.

Finally, recalling the definitions of $N_1$ and $N_2$ in Lemma \ref{stuck}, let us define the time $N:=\max (N_1,N_2, N_3, (1+\al)/\eps )$. We work at large times $k\ge N$, and the rest of the
proof is only combinatorial.

By definition of $R'$, there exists $k_1\ge N$ such that $X_{k_1}=x+K+1$. Therefore, using Lemma \ref{stuck}, for all
$j\in\{1,\ldots,K\}$,
\[
\frac{\Delta_{k_1}(x+j)}{k_1}>-2\eps.
\]
Let us prove by induction on $j=1,\ldots,K$, that, for all $k\ge k_1$, we have
\[
\frac{\Delta_{k}(x+j)}{k}>-3\eps M^{j(j-1)/2}.
\]

\begin{longlist}[(ii)]
\item[(i)] \emph{The base case}: $j=1$. Assume by contradiction that there exists $\kc\ge k_1$ such that
\[
\frac{\Delta_{\kc}(x+1)}{\kc}\le-3\eps,
\]
and define
\[
\kt:=\max \biggl\{k\in[k_1,\kc]\dvtx \frac{\Delta_k(x+1)}{k}>-2\eps
\biggr\}.
\]
Hence, using Lemma \ref{stuck} between $\kt$ and $\kc$, the walk remains left-hand from $x+1$, and thus stays confined to
$\{x,x+1\}$, therefore,
\[
\frac{\Delta_{\kc}(x+1)}{\kc}=\frac{\Delta_{\kt}(x+1)+(\kc-\kt)}{\kt+(\kc-\kt)}>-2\eps,
\]
which contradicts our assumption.

\item[(ii)] \emph{Induction step}. Assume that $j\in\{1,\ldots,K-1\}$, and that, for all $i\in\{1,\ldots,j\}$ and $k\ge k_1$, we
have
\begin{equation}
\label{hyprec1} \frac{\Delta_k(x+i)}{k}>-3\eps M^{i(i-1)/2}.
\end{equation}
We want to show that, for all $k\ge k_1$,
\[
\frac{\Delta_k(x+j+1)}{k}>-3\eps M^{j(j+1)/2}.
\]
By contradiction, assume that there exists a finite integer $k^*$ such that
\begin{equation}
\label{defkstar} k^*:=\min \biggl\{k\ge k_1\dvtx \frac{\Delta_{k}(x+j+1)}{k}
\le-3 \eps M^{j(j+1)/2} \biggr\}<\infty.
\end{equation}
By definition of $R'$, there exists $k_2>k^*$ such that $X_{k_2}=x+K+1$, and, using Lemma \ref{stuck},
\[
\frac{\Delta_{k_2}(x+j+1)}{k_2}>-2\eps.
\]
Define the following times:
\begin{eqnarray}
\ks&:=&\max_k \biggl\{k_1\le k \le k^*\dvtx
\frac{\Delta_k(x+j+1)}{k}>-3\eps M^{j(j-1)/2} \biggr\},\label{defks}
\\
\kb&:=&\min_k \biggl\{k^*\le k \le k_2\dvtx
\frac{\Delta_k(x+j+1)}{k}>-3\eps M^{j(j-1)/2} \biggr\}.\label{kb}
\end{eqnarray}

\begin{figure}

\includegraphics{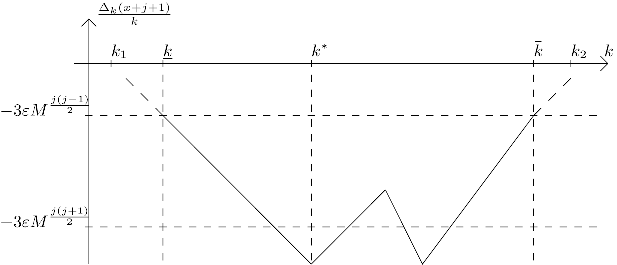}

\caption{Local stream at $x+j+1$.}\label{dessin2}
\end{figure}
Note that the definitions of $k_1$, $\ks$, $k^*$, $\kb$ and $k_2$ will be kept throughout the proof. All other
notation used to define some other times change from part to part.

We need the following lemma, which will be shown right after this proof.
\end{longlist}

\begin{lem}\label{deltaki}
Under the previous assumptions, in particular assuming \eqref{hyprec1} and \eqref{defkstar}, for all $i\in\{1,\ldots,j\}$,
and for all $k$ such that $\ks\le k \le\kb$, we have
\[
\biggl\llvert \frac{\Delta_k(x+i)}{k}\biggr\rrvert <3\eps M^{(j(j-1)/2)+(j-i)}.
\]
\end{lem}

Assuming that \eqref{defkstar} holds, the times $\ks, k^*, \kb$ are key-times at which we know the behavior of the
local stream at $x+j+1$; see Figure~\ref{dessin2}. Using Lemmas \ref{deltaki}~and~\ref{stuck2}, we will compare the evolution of the
local times on the interval $\{x,\ldots,x+j+1\}$ to the solution of the system $(E_j)$, defined in Section~\ref{trigo}.

This will enable us to emphasize some contradictions about the evolution of the local streams and conclude that \eqref{defkstar} cannot hold.

Let us distinguish between two cases: $j\le L-1$ and $j\ge L$.

First, assume that $j\le L-1$. By Lemma \ref{stuck}, the walker remains left-hand from $x+j+1$ on $[\ks,k^*]$, thus the walk is confined to $\{x,\ldots,x+j+1\}$ on this time-span.

Moreover, notice that, at each time-step, $\Delta_k(x+j+1)$ varies by $0$, $\pm1$ or $\pm\al$. Using the definitions \eqref{defks} of $\ks$ and \eqref{defkstar} of $k^*$, and using \eqref{defMeps}, this yields
\begin{eqnarray}\label{inegk}
\frac{k^*-\ks}{k^*}&>&\frac{1}{1+\al}\times\frac{-\Delta_{k^*}(x+j+1)+\Delta_{\ks}(x+j+1)}{k^*}
\nonumber
\\
&>&\frac{1}{1+\al} \bigl(3\eps M^{j(j+1)/2}-3\eps M^{j(j-1)/2} \bigr)
\ge \frac{3\eps M^{(j(j+1)/2)-1}}{1+\al} (M-1 )\hspace*{-20pt}
\\
&>&\frac{6C_1\eps M^{(j(j+1)/2)-1}}{d}.\nonumber
\end{eqnarray}

Hence, using \eqref{LB}, deduced from Lemma \ref{stuck2}, with $i_1=x$, $i_2=x+j$, $k=\ks$ and $k'=k^*$, and using Lemma \ref{deltaki}, this yields
\begin{eqnarray*}
\frac{\Delta_{k^*}(x+j+1)}{k^*}&>&\frac{\Delta_{\ks}(x+j+1)}{\ks}\times\frac{\ks}{k^*}>-3\eps
M^{j(j-1)/2},
\end{eqnarray*}
where we use the definition \eqref{defks} of $\ks$ for the last inequality. This contradicts the definition \eqref{defkstar} of $k^*$ and concludes the first case $j\le L-1$.

Now, assume that $j\ge L$. Similarly, by Lemma \ref{stuck}, the walk is
confined to $\{x,\ldots,x+j+1\}$ on the time-span $[k^*,\kb]$.
Recall the definition \eqref{defkstar} of $k^*$, \eqref{defMeps},
and that $d>3\eps M^{j(j+1)/2}$. Then, using the upper bound of Lemma~\ref{stuck2} with $i_1=x$, $i_2=x+j$, $k=k^*$ and $k'=\kb$, and using Lemma \ref{deltaki}, this yields
\begin{eqnarray*}
\frac{\Delta_{\kb}(x+j+1)}{\kb}&<& -3\eps M^{j(j+1)/2}+6C_1\eps
M^{(j(j+1)/2)-1}
\\
&<&-3\eps M^{j(j-1)/2}.
\end{eqnarray*}
This contradicts the definition \eqref{kb} of $\kb$ and finishes the second case, the induction step and the proof of
Proposition \ref{lemdeltak1}.
\end{pf*}

\begin{pf*}{Proof of Lemma \ref{deltaki}}
By assumption \eqref{hyprec1}, we know that, for all $i\in\{1,\ldots,j\}$ and for all $k\ge k_1$,
\[
\frac{\Delta_k(x+i)}{k}>-3\eps M^{i(i-1)/2}\ge -3\eps M^{(j(j-1)/2)+(j-i)}.
\]
Thus, it remains to prove that, for all $k\in[\ks,\kb]$,
\begin{eqnarray}
\label{indi} \frac{\Delta_k(x+i)}{k}<3\eps M^{(j(j-1)/2)+(j-i)}.
\end{eqnarray}
We will proceed in the same way as in the previous proof, by induction for $i=j,\ldots,1$.

\begin{longlist}[(ii)]
\item[(i)] \emph{The base case}: $i=j$. Notice that $\Delta_{\ks}(x+j)/\ks<2\eps$ as $\Delta_{\ks+1}(x+j)/(\ks+1)<0$ by definition \eqref{defks} of $\ks$.

Now, let us derive this result on the whole time-span $[\ks,\kb]$, by proving that, for all $k\in[\ks,\kb]$,
$\Delta_{k}(x+j)/k<3\eps$. By contradiction, assume that there exists $\kc\in(\ks,\kb]$, such that
\[
\frac{\Delta_{\kc}(x+j)}{\kc}\ge3\eps,
\]
and let $\kt$ the last time before $\kc$ such that $\Delta_{\kt}(x+j)/\kt<2\eps$, that is,
\[
\kt:=\max \biggl\{k\in[\ks,\kc]\dvtx \frac{\Delta_{\kt}(x+j)}{\kt}<2\eps \biggr\}.
\]
Thus, using Lemma \ref{stuck}, the walk is confined to $\{j,j+1\}$ on the time-span $[\kt,\kc]$. It implies
\[
\frac{\Delta_{\kc}(x+j)}{\kc}=\frac{\Delta_{\kt}(x+j)-(\kc-\kt)}{\kt+(\kc-\kt)}<2\eps,
\]
which contradicts our assumptions.\vspace*{1pt}

Therefore, for all $k\in[\ks,\kb]$,
\[
\frac{\Delta_k(x+j)}{k}<3\eps\le3\eps M^{j(j-1)/2}.
\]

\item[(ii)] \emph{The induction step}. Assume that, for all $m\in\{i+1,\ldots,j\}$, $i\ge1$, and for all $k\in[\ks,\kb]$,
\begin{eqnarray}
\label{hyp} \biggl\llvert \frac{\Delta_k(x+m)}{k} \biggr\rrvert <3\eps
M^{(j(j-1)/2)+(j-m)}.
\end{eqnarray}
Let us prove this inequality for $m=i$. More precisely, using \eqref{hyprec1}, it remains to prove that:
\begin{eqnarray}
\label{etap2} \frac{\Delta_{k}(x+i)}{k}<3\eps M^{(j(j-1)/2)+(j-i)}\qquad\mbox{for all } k
\in[\ks,\kb].
\end{eqnarray}

We will now use Lemma \ref{stuck2}, distinguishing between two cases: $j-i\le L-1$ and $j-i\ge L$.

\begin{longlist}[(ii.a)]
\item[(ii.a)] \emph{Assume that $j-i\le L-1$}.

Let us start by proving that
\begin{eqnarray}
\label{etap1} \frac{\Delta_{\ks}(x+i)}{\ks}<\frac{3}{2}\eps M^{(j(j-1)/2)+(j-i)},
\end{eqnarray}
and we will then prove \eqref{etap2} by deriving this property on the whole time-span $[\ks,\kb]$.

By contradiction, assume that
\[
\frac{\Delta_{\ks}(x+i)}{\ks}\ge \frac{3}{2}\eps M^{(j(j-1)/2)+(j-i)}.
\]
Note that, using Lemma \ref{stuck} and recalling the definition \eqref{defks} of $\ks$ and \eqref{kb} of $\kb$, the
walk is confined to $\{x+i,\ldots,x+j+1\}$ as long as $\Delta_k(x+i)/k\ge \eps$ and $k\le\kb$. As
$\llvert  \Delta_{k+1}(x+i)-\Delta_k(x+i)\rrvert  \le 1+\al$, the walk stays on the right of $x+i$ at least until the time $\kc$, which is
the integer defined by
\begin{eqnarray}
\label{def1kc} \kc:= \biggl\lfloor\frac{1+\al +(3/2)\eps M^{(j(j-1)/2)+(j-i)}}{1+\al+\eps}\ks \biggr\rfloor,
\end{eqnarray}
where $\lfloor\cdot\rfloor$ is the floor function.
Using \eqref{inegk}, using that $\ks>1/\eps$ and $d<1$, let us prove that $\kc\in(\ks,k^*)$:
\begin{eqnarray*}
k^*-\kc&=& \bigl(k^*-\ks \bigr)-(\kc-\ks)
\\
&>& \ks \bigl(6C_1\eps M^{(j(j+1)/2)-1}-\tfrac{3}{2}\eps
M^{(j(j+1)/2)-i} \bigr)+\eps\ks-1>0.
\end{eqnarray*}

Hence,\vspace*{2pt} using Lemma \ref{stuck}, the walk is confined to $\{x+i,\ldots,x+j+1\}$ on
the time-span $[\ks,\kc]$.
Moreover, using \eqref{def1kc} and \eqref{defMeps},
\begin{eqnarray*}
\frac{\kc-\ks}{\kc}&>& \frac{(3/2)\eps M^{(j(j-1)/2)+(j-i)}-\eps}{1+\al+(3/2)\eps M^{(j(j-1)/2)+(j-i)}} > \frac{(3/2)\eps M^{(j(j-1)/2)+(j-i)}-\eps}{2(1+\al)}
\\
&>& \frac{6C_1\eps M^{(j(j-1)/2)+(j-i)-1}}{d} \biggl(\frac{M\times d}{4C_1(1+\al)}-\frac{d}{12C_1(1+\al)} \biggr)
\\
&>& \frac{6C_1\eps M^{(j(j-1)/2)+(j-i)-1}}{d}.
\end{eqnarray*}

Thus, using \eqref{LB} with $i_1=x+i$, $i_2=x+j$, $k=\ks$ and $k'=\kc$, and using the hypothesis \eqref{hyp}, we obtain
\begin{eqnarray*}
\frac{\Delta_{\kc}(x+j+1)}{\kc}&>&\frac{\Delta_{\ks}(x+j+1)}{\ks}\times\frac{\ks}{\kc}>-3\eps
M^{j(j-1)/2},
\end{eqnarray*}
which contradicts the definitions \eqref{defks} of $\ks$ and \eqref{defkstar} of $k^*$, and we can conclude that
\eqref{etap1} holds.

Now, let us derive \eqref{etap1} for the whole time-span $[\ks,\kb]$: in other words, let us prove \eqref{etap2}, which
states that, for all $k\in[\ks,\kb]$,
\begin{eqnarray*}
\frac{\Delta_{k}(x+i)}{k}<3\eps M^{(j(j-1)/2)+(j-i)}.
\end{eqnarray*}
By contradiction, assume that there exists $\kc\in[\ks,\kb]$ such that
\begin{eqnarray}
\label{defkc1} \frac{\Delta_{\kc}(x+i)}{\kc}\ge3\eps M^{(j(j-1)/2)+(j-i)},
\end{eqnarray}
and define $\kt$ such that
\begin{eqnarray}
\label{def2kt} \kt:=\max \biggl\{k\in[\ks,\kc]\dvtx \frac{\Delta_{\kt}(x+i)}{\kt}<
\frac{3}{2}\eps M^{(j(j-1)/2)+(j-i)} \biggr\}.
\end{eqnarray}
Using that $|\Delta_{k+1}(x+i)-\Delta_k(x+i)|\le1+\al$, we have that
\begin{eqnarray*}
\frac{\kc-\kt}{\kc}&>&\frac{1}{1+\al} \biggl(3\eps M^{(j(j-1)/2)+(j-i)}-
\frac{3}{2}\eps M^{(j(j-1)/2)+(j-i)} \biggr)
\\
&>& 3\eps M^{(j(j-1)/2)+(j-i)-1}\times \frac{M}{2(1+\al)}>\frac{6C_1\eps M^{(j(j-1)/2)+(j-i)-1}}{d}.
\end{eqnarray*}
Using Lemma \ref{stuck}, the walk is confined to $\{x+i,\ldots,x+j+1\}$ on the time-span $[\kt,\kc]$. Using \eqref{UB}, deduced from Lemma \ref{stuck2}, with $i_1=x+i$, $i_2=x+j$, $k=\kt$ and $k'=\kc$, and using the hypothesis \eqref{hyp}, we obtain
\begin{eqnarray*}
\frac{\Delta_{\kc}(x+i)}{\kc}&<&\frac{\Delta_{\kt}(x+i)}{\kt}\times\frac{\kt}{\kc}<
\frac{3}{2}\eps M^{(j(j-1)/2)+(j-i)},
\end{eqnarray*}
which contradicts assumption \eqref{defkc1}, proves \eqref{etap2} and finishes the first case $j-i\le L-1$.

\item[(ii.b)] \emph{Assume that  $j-i\ge L$}.

Here, we directly prove \eqref{etap2}. By contradiction, assume that there exists $\kt\in[\ks,\kb]$ such that
\begin{eqnarray}
\label{ptptpt} \frac{\Delta_{\kt}(x+i)}{\kt}\ge3\eps M^{(j(j-1)/2)+(j-i)},
\end{eqnarray}
which implies, by Lemma \ref{stuck}, that $X_{\kt}\in\{x+i,\ldots,x+j+1\}$. Let us prove that the walker stays stuck
forever in this interval. In other words, let us prove by induction that $X_{\kc}\in\{x+i,\ldots,x+j+1\}$ for all
$\kc\ge\kt$. This will contradict the definition of $R'$ and enable us to conclude.

Assume, for $\kc\ge\kt$, that $X_m\in\{x+i,\ldots,x+j+1\}$ for all $m\in[\kt,\kc]$, and let us prove that
$X_{\kc+1}\in\{x+i,\ldots,x+j+1\}$.

On one hand, using that $d>3\eps M^{K(K+1)}$ by \eqref{defMeps}, using \eqref{ptptpt}, and applying Lemma \ref{stuck2}, with $i_1=x+i$, $i_2=x+j$, $k=\kt$ and $k'=\kc$, with the hypothesis~\eqref{hyp}, we obtain
\begin{eqnarray*}
\frac{\Delta_{\kc}(x+i)}{\kc}&\ge&3\eps M^{(j(j-1)/2)+(j-i)}-6C_1\eps
M^{(j(j-1)/2)+(j-i)-1}> \eps.
\end{eqnarray*}
Thus, using Lemma \ref{stuck}, if $X_{\kc}=x+i$ then $X_{\kc+1}=x+i+1$.

On the other hand, if $\kc\le k^*$, then $\Delta_{\kc}(x+j+1)/\kc<-\eps$ and, therefore, by Lemma \ref{stuck}, if
$X_{\kc}=x+j+1$, then $X_{\kc+1}=x+j$. Otherwise, if $\kc>k^*$, applying Lemma \ref{stuck2}, with $i_1=x+i$, $i_2=x+j$, $k=k^*$ and $k'=\kc$, with the hypothesis~\eqref{hyp}, and using \eqref{defkstar} and \eqref{defMeps}, we obtain
\begin{eqnarray*}
\frac{\Delta_{\kc}(x+j+1)}{\kc}&<&-3\eps M^{j(j+1)/2}+6C_1\eps
M^{(j(j-1)/2)+(j-i)-1}<-3\eps M^{j(j-1)/2}.
\end{eqnarray*}

Thus,\vspace*{1pt} again using Lemma \ref{stuck}, if $X_{\kc}=x+j+1$ then $X_{\kc+1}=x+j$. Therefore, by induction, this implies
that $X_{\kc}\in\{x+i,\ldots,x+j+1\}$ for all $\kc\ge\kt$, which contradicts the definition of $R'$ and concludes the
second case, $j-i\ge L$, and the proof of the lemma.\quad\qed
\end{longlist}
\end{longlist}\noqed
\end{pf*}


\subsection{Proof of Lemma \texorpdfstring{\protect\ref{somfinie}}{2.4}} \label{aswei2}


Fix $L\ge1$ and $\al\in(\al_{L+1},\al_L)$. Let $x,K\in\mathbb{Z}$ such that $K\ge L$, and fix $a>0$. Recall that Lemma \ref{somfinie} states that, almost surely,
\begin{eqnarray*}
\bigl\{R'=\{x,x+1,\ldots,x+K+1\} \bigr\}\subset \Biggl\{\sum
_{k=1}^{+\infty} e^{a\beta[l_k(x+L+2)-\al
l_k(x+L+1)]}<+\infty \Biggr
\}.
\end{eqnarray*}

\begin{pf*}{Proof of Lemma \ref{somfinie}}
Let us work on the event $ \{R'=\{x,x+1,\ldots,x+K+1\} \}$, where $K\ge L$.

First, note that there exists a.s. $N_1\in\mathbb{N}$ such that, for all $k\ge N_1$, $X_k\in R'$.
Then we use Proposition \ref{lemdeltak1}, which implies that for all $\eps>0$, there a.s. exists $N_2\in\mathbb{N}$
such that for all $j\in\{x+1,\ldots,x+K\}$ and for all $k\ge N_2$,
\[
\biggl\llvert \frac{\Delta_k(j)}{k}\biggr\rrvert <\frac{\eps}{4}.
\]
Define $N:=\max (N_1, N_2 )$.

We\vspace*{1.5pt} need the results on the linear system presented in Proposition \ref{system}. Notice that, if
$(\widetilde{l_0},\ldots,\tilde{l}_{K+2})$ is such that $\widetilde{l_0}=0$, $\sum_{j=1}^{K+2}\tilde{l}_j=1$, and
$|\widetilde{d_j}|<\eps$ for all $j=1,\ldots,K$, where $\widetilde{d_j}$ is defined as in \eqref{dj}, then this vector can
be seen as a perturbed solution of the linear system (\ref{eq}). Thus, by Proposition \ref{system},
$\tilde{l}_{L+2}-\al \tilde{l}_{L+1}<-c(K)/2$ as soon as $\eps$ is small enough, depending on $K$ and $\al$.

Then, for any $k\ge 2N$, define for all $j\in\{0,\ldots,K+2\}$,
\[
\tilde{l}_j=\frac{l_k(x+j)-l_N(x+j)}{k-N}.
\]
Then we have $\tilde{l}_0=\tilde{l}_{K+2}=0$, $\sum_{j=1}^{K+2}\tilde{l}_j=1$ and, for all
$j\in\{1,\ldots,K\}$,
\[
\llvert \widetilde{d}_j\rrvert \le \frac{\llvert  \Delta_k(x+j)\rrvert  +\llvert  \Delta_N(x+j)\rrvert  }{k}\times
\frac{k}{k-N}<\eps.
\]
This implies, as soon as $\eps>0$ is small enough, that
\[
\frac{ (l_k(L+2)-l_N(L+2) )-\al ( l_k(L+1)-l_N(L+1) )}{k-N}<-\frac{c(K)}{2},
\]
therefore, for any $a>0$, we can complete the proof with
\begin{eqnarray*}
&& \sum_{k=2N}^{+\infty} e^{a\beta[l_k(L+2)-\al l_k(L+1)]}
\\
&&\qquad \le
e^{a\beta[l_{N}(L+2)-\al
l_{N}(L+1)]}\times\sum_{k=2N}^{+\infty}
e^{-a(c(K)/2)(k-N)}<+\infty.
\end{eqnarray*}\upqed
\end{pf*}

\section{Proof of Theorem \texorpdfstring{\protect\ref{algrd}}{1.5}} \label{palgrd}


The goal of this section is to prove that if $\al>1=\al_2$, then the walk localizes on $3$ vertices almost surely.

\begin{pf*}{Proof of Theorem \protect\ref{algrd}}
Fix $x\in\mathbb{Z}$. We want to prove that $\Pb (R'=\{x,x+1,x+2,x+3\} )=0$. We will then conclude by taking the union over $x$ and using Theorem \ref{L2L3}.

Let us first prove the following lemma.

\begin{lem}\label{infkc}
For any constant $C>0$, on the event $\{R'=x,x+1,x+2,x+3\}$, there a.s. exist infinitely many $k$'s such that one of
the following statements holds:
\begin{itemize}
\item $X_k=x+1$, $\Delta_k(x+1)\le -C$ and $\Delta_k(x+1)+\Delta_k(x+2)\le -C$;
\item $X_k=x+2$, $\Delta_k(x+2)\ge C$ and $\Delta_k(x+1)+\Delta_k(x+2)\ge C$.
\end{itemize}
Let $\mathcal{K}$ be the set of those $k$'s.
\end{lem}

\begin{pf}
Let us work on the event $\{R'=\{x,x+1,x+2,x+3\}\}$.

Using Proposition \ref{lemdeltak1}, and solving equation \eqref{eq} in the case $L=2$, one can deduce that $Z_k(x+1)/Z_k(x)$
converges to $\alpha+2$ (see Proposition \ref{systpos}).

Using the conditional Borel--Cantelli lemma [see \cite{Durrett}, Chapter~4, (4.11)], we have, a.s.,
\[
\lim_{m\rightarrow\infty} \frac{\sum_{k=1}^m\1_{\{X_k=x+1\}}\Pb(X_{k+1}=x|\Fk)}{Z_m(x)}=1.
\]

This implies that, for any $\eps>0$, there exist infinitely many $k$'s such that $X_k=x+1$ and
$\Pb(X_{k+1}=x|\Fk)>\frac{1}{\al+2}-\eps$, and there exist infinitely many $k$'s such that $X_k=x+1$ and
$\Pb(X_{k+1}=x|\Fk)<\frac{1}{\al+2}+\eps$. Otherwise, this would contradict the conditional Borel--Cantelli lemma.

Then, there exists a positive constant $c_{\max}$ (resp., $c_{\min}$), depending only on $\alpha$ and $\beta$, such that
$X_k=x+1$ and $\Delta_k(x+1)<c_{\max}$ [resp., $\Delta_k(x+1)>-c_{\min}$] for infinitely many $k$'s.

If $k$ is large enough, then the walk does not leave the set $\{x,\ldots,x+3\}$. Therefore, $\Delta_k(x+1)$ decreases at
most by $2\beta$ between two times where the walker is in $x+1$. Hence, there exist infinitely many $k$'s such that
$X_k=x+1$ and
\[
-c_{\min}\le \Delta_k(x+1)\le c_{\max}.
\]
 Notice that we could give some explicit bounds $c_{\min}$ and $c_{\max}$ but it is not useful.

Now, notice that, on the event $\{R'=\{x,x+1,x+2,x+3\}\}$, $\lim_k \Delta_k(x)=+\infty$. Therefore, on this event,
there exist infinitely many $k$'s such that: $-c_{\min}\le \Delta_k(x+1)\le c_{\max}$, $X_k=x+1$ and $\Delta_k(x)\ge
C$, for any constant $C>0$.

Then assume that
\[
k_0\in \bigl\{k\in\mathbb{N}\dvtx -c_{\min}\le
\Delta_k(x+1)\le c_{\max}, X_k=x+1,
\Delta_k(x)\ge C \bigr\},
\]
 and let us prove that
\begin{eqnarray}
\label{labproba} \Pb \bigl(\exists k\in[k_0,M+k_0]\dvtx
k\in \mathcal{K}|\mathcal{F}_{k_0} \bigr)\ge\delta>0,
\end{eqnarray}
where $M:=M(C,\al,\beta)$ and $\delta:=\delta(C,\al,\beta)$ are two positive constants.

To complete this proof, let us consider two cases: $\Delta_{k_0}(x+1)+\Delta_{k_0}(x+2)< C$ and
$\Delta_{k_0}(x+1)+\Delta_{k_0}(x+2)\ge C$. Let us show, in each case, that \eqref{labproba} holds. This will imply the
conclusion, by applying again the conditional Borel--Cantelli lemma.

\begin{longlist}[(ii)]
\item[(i)] Assume that $\Delta_{k_0}(x+1)+\Delta_{k_0}(x+2)< C$.
As $-c_{\min}\le \Delta_{k_0}(x+1)\le c_{\max}$, this implies that $\Delta_{k_0}(x+2)<c_{\min}+C$. Then we have some
control on the conditional probabilities to go from $x$ to $x+1$, from $x+1$ to $x$ and $x+2$, and to go from $x+2$ to
$x+1$.
This implies that we can control the probabilities of the following trajectories: the walker starts  at $x+1$ at time $k_0$; he crosses, say $2T$ times, the edge $\{x,x+1\}$, coming back to $x+1$; then the walker crosses roughly $2T+C+c_{\max}$ times the edge $\{x+1,x+2\}$, coming back to $x+1$ at some time~$k_1$. Note that, on these trajectories, the walker only visits the sites $x$, $x+1$ and $x+2$ between the times $k_0$ and $k_1$.

We deduce that, with probability bounded below by a constant depending only on $C$, $\al$ and $\beta$, there exists
\[
k_1\in \biggl[k_0,k_0+C+c_{\max}+4
\biggl(1+\frac{C}{\al-1} \biggr) \biggr],
\]
such that
$X_{k_1}=x+1$ and
\begin{eqnarray*}
l_{k_1}(x+1)-l_{k_0}(x+1)&\ge&2C/(\al-1),
\\
l_{k_1}(x+2)-l_{k_0}(x+2)&\ge& l_{k_1}(x+1)-l_{k_0}(x+1)+C+c_{\max},
\\
l_{k_1}(x)-l_{k_0}(x)&=&l_{k_1}(x+3)-l_{k_0}(x+3)=l_{k_1}(x+4)-l_{k_0}(x+4)=0.
\end{eqnarray*}
Moreover, notice that, for all $k$,
\begin{eqnarray}\label{somdelta}
&& \Delta_k(x+1)+\Delta_k(x+2)
\nonumber\\[-8pt]\\[-8pt]\nonumber
&&\qquad =-(\alpha-1)
\bigl(l_k(x+1)-l_k(x+3) \bigr)
-\al l_k(x)+\al l_k(x+4).
\end{eqnarray}
We can now conclude that $\Delta_{k_1}(x+1)\le-C$ and $\Delta_{k_1}(x+1)+\Delta_{k_1}(x+2)\le-C$.

\item[(ii)] Assume that $\Delta_{k_0}(x+1)+\Delta_{k_0}(x+2)\ge C$ and recall that $-c_{\min}\le \Delta_{k_0}(x+1)\le
c_{\max}$. Let us prove that, with probability bounded below by a constant depending only on $C$, $\al$ and $\beta$, there exists $k_1\ge k_0$ such that $X_{k_1}=x+2$,
$\Delta_{k_1}(x+2)>C$ and $\Delta_{k_1}(x+1)+\Delta_{k_1}(x+2)\ge C$.

First notice, by \eqref{somdelta}, that the visits to the edge $\{x+1,x+2\}$ do not change the value of
$\Delta_{k_0}(x+1)+\Delta_{k_0}(x+2)$.

If $\Delta_{k_0}(x+2)>C$, then with probability bounded below by a constant depending only on $C$, $\al$ and $\beta$, $k_1=k_0+1$ because the walker just needs to jump from
$x+1$ to $x+2$. Otherwise, $C-c_{\max}\le\Delta_{k_0}(x+2)\le C$ and we have some control on the conditional
probabilities to jump from $x+1$ to $x+2$ and from $x+2$ to $x+1$, and we can lower-bound the probability to cross only
the edge $\{x+1,x+2\}$, at least $c_{\max}$ times. Thus, with probability bounded below by a constant depending only on $C$, $\al$ and $\beta$, there exists $k_1\in
[k_0,k_0+c_{\max}+2]$ such that $X_{k_1}=x+2$, $\Delta_{k_1}(x+2)>C$ and $\Delta_{k_1}(x+1)+\Delta_{k_1}(x+2)\ge C$.\quad\qed
\end{longlist}\noqed
\end{pf}

Let us define by induction a nondecreasing sequence of stopping times. First, define the time
$\tau_0:=\inf_k\{k\ge0\dvtx k\in\mathcal{K}\}$, then define, for all $n\ge0$:
\begin{eqnarray*}
\tau_{2n+1}&:=&\inf_k \bigl\{k\ge
\tau_{2n}\dvtx X_k=(x+3)\1_{\{X_{\tau_{2n}}=x+1\}}+ x
\1_{\{X_{\tau_{2n}}=x+2\}} \bigr\},
\\
\tau_{2(n+1)}&:=&\inf_k\{k\ge\tau_{2n+1}
\dvtx X_k \in\mathcal{K}\}.
\end{eqnarray*}

Notice that
\[
\bigl\{R'=\{x,x+1,x+2,x+3\} \bigr\}\subset\bigcap
_n \{\tau_n<+\infty \}.
\]
Let us prove that
\[
\Pb (\tau_{2n+1}<\infty|\mathcal{F}_{\tau_{2n}} )
\1_{\{\tau_{2n}<+\infty\}}<1-\delta,
\]
for some $\delta>0$, which depends only on $\al$, $\beta$ and $C$. This will enable us to conclude.

First, assume that some vector $(l_0,l_1,l_2,l_3)$ is such that $l_0\ge0$, $l_3=0$, $-\al
l_0+l_1-l_2=d_1\in\mathbb{R}$, and $l_1+l_2=1$. Define $d_2:=-\al l_1+l_2$, then one can easily compute $d_2$ as a
function of $d_1$ and $l_0$, and prove that, if $l_0\ge0$,
\begin{eqnarray}
\label{minisyst} d_1+d_2=-\frac{\al-1}{2}-
\frac{\al-1}{2}d_1-\al\frac{1+\al}{2}l_0\le -
\frac{\al-1}{2}-\frac{\al-1}{2}d_1.
\end{eqnarray}
Note that this is a particular case of Corollary \ref{corgensyst}.

Now, assume that $\{\tau_{2n}<+\infty\}$, and recall that $\tau_{2n}\in\mathcal{K}$. We will complete the proof assuming
that $X_{\tau_{2n}}=x+1$, $\Delta_{\tau_{2n}}(x+1)\le -C$ and $\Delta_{\tau_{2n}}(x+1)+\Delta_{\tau_{2n}}(x+2)\le -C$.
The other case, if $X_{\tau_{2n}}=x+2$, is the exact symmetric of this one.

For any $k>\tau_{2n}$, as long as $l_k(x+3)-l_{\tau_{2n}}(x+3)=0$, defining for all $j\in\{0,1,2,3\}$
\[
l_j=\frac{l_k(x+j)-l_{\tau_{2n}}(x+j)}{l_k(x+1)-l_{\tau_{2n}}(x+1)+l_k(x+2)-l_{\tau_{2n}}(x+2)},
\]
we deduce, using \eqref{minisyst}, that
\begin{eqnarray*}
\Delta_k(x+1)+\Delta_k(x+2)&=&\Delta_k(x+1)+
\Delta_k(x+2)
\\
&&{}-\Delta_{\tau_{2n}}(x+1)-\Delta_{\tau_{2n}}(x+2)
\\
&&{}+\Delta_{\tau_{2n}}(x+1)+\Delta_{\tau_{2n}}(x+2)
\\
&\le&-\frac{\alpha-1}{2} \bigl(l_k(x+1)+l_k(x+2) \bigr)
\\
&&{}-\frac{\alpha-1}{2} \bigl(-l_{\tau_{2n}}(x+1)-l_{\tau_{2n}}(x+2)
\bigr)
\\
&&{}-\frac{\alpha-1}{2} \bigl(\Delta_k(x+1)-\Delta_{\tau_{2n}}(x+1)
\bigr)-C.
\end{eqnarray*}
This inequality also holds when $k=\tau_{2n}$.
Thus, for any $k\ge\tau_{2n}$, if
\begin{eqnarray*}
&& \Delta_k(x+1)-\Delta_{\tau_{2n}}(x+1)
\\
&&\qquad >-\tfrac{1}{2}
\bigl(l_k(x+1)+l_k(x+2)-l_{\tau_{2n}}(x+1)-l_{\tau_{2n}}(x+2)
\bigr),
\end{eqnarray*}
then
\begin{eqnarray*}
&& \Delta_k(x+1)+\Delta_k(x+2)
\\
&&\qquad <-\frac{\alpha-1}{4}
\bigl(l_k(x+1)+l_k(x+2)-l_{\tau_{2n}}(x+1)-l_{\tau_{2n}}(x+2)
\bigr)-C,
\end{eqnarray*}
otherwise, as $\Delta_{\tau_{2n}}(x+1)\le-C$,
\[
\Delta_k(x+1)\le -\tfrac{1}{2} \bigl(l_k(x+1)+l_k(x+2)-l_{\tau_{2n}}(x+1)-l_{\tau_{2n}}(x+2)
\bigr)-C.
\]
To conclude, we just have to notice that
\begin{eqnarray*}
\hspace*{-4pt}&&\Pb (\tau_{2n+1}<\infty|\mathcal{F}_{\tau_{2n}} )
\1_{\{\tau_{2n}<+\infty\}}
\\
\hspace*{-4pt}&&\!\qquad\le \sum_{k=\tau_{2n}}^{\infty}\E \biggl(
\frac{\1_{\{X_k=x+1,l_k(x+3)=l_{\tau_{2n}}(x+3)\}}}{1+e^{-2\beta\Delta_k(x+1)}}\frac{1}{1+e^{-2\beta\Delta_k(x+2)}}\bigg| \mathcal{F}_{\tau_{2n}} \biggr)
\\
\hspace*{-4pt}&&\!\qquad\le \sum_{k=\tau_{2n}}^\infty\E \bigl(
\1_{\{X_k=x+1,l_k(x+3)=l_{\tau_{2n}}(x+3)\}}e^{2\beta\Delta_k(x+1)}\wedge e^{2\beta[\Delta_k(x+1)+\Delta_k(x+2)]}| \mathcal{F}_{\tau_{2n}}
\bigr)
\\
\hspace*{-4pt}&&\!\qquad\le \sum_{k=0}^\infty
e^{-2\beta[((\alpha-1)/4\wedge 1/2)k+C]}
\le\frac{e^{2\beta((\alpha-1)/4 \wedge 1/2)}}{2\beta((\alpha-1)/4 \wedge 1/2)} e^{-C}
\le 1-\delta,
\end{eqnarray*}
choosing $C=C(\alpha,\beta)$ large enough and for some $\delta=\delta(\alpha,\beta)>0$. This completes the proof.
\end{pf*}


\section{Generalizations} \label{conclusion}


In \cite{ETW}, interesting generalizations are proposed, including a model where the interaction depends on more
than four neighbouring edges, that is, replacing the definition of $\Delta_n(j)$ by
\[
\al_k l_n(j-k)+\cdots+\al_1
l_n(j-1)+\al_0 l_n(j)-\al_0
l_n(j+1)-\cdots-\al_k l_n(j+1+k),
\]
for some $\al_0,\ldots,\al_k$. The techniques of \cite{ETW}, as it is noticed therein, seem to be adaptable to derive
results about such models, by investigating the behavior of the associated generalized Fibonacci sequence. Moreover, in
some cases, we believe that our techniques are also adaptable. In particular, we could adapt the variant of Rubin's
construction as it is done in the present paper and apply it to these models, as soon as $\al_1,\ldots,\al_k$ are
nonpositive. The parameter $\al_0$ could be positive. In these cases, we could recover some sort of weight function,
which could be decreasing with respect to the local times on neighbouring edges but nondecreasing with respect to
local times on edges further away.
More precisely, one could prove a generalized version of Lemma \ref{lemPT}
and prove that all the results  of Section~\ref{monomart} hold for a process $X'$ as soon as it is described by the following transition probability:
\begin{eqnarray*}
&&\Pb \bigl(X_{k+1}'=X_k'-1|\Fk
\bigr)
\\
&&\qquad= \frac{\widetilde{w}(Z_k(X_k'
-1),\ldots,Z_k(X_k'-n_0))}{\widetilde{f}^-(X_k',N_k(X_k',X_k'-1))}
\\
&& \quad\qquad{}\Big/ \biggl(\frac{\widetilde{w}(Z_k(X_k'
-1),\ldots,Z_k(X_k'-n_0))}{\widetilde{f}^-(X_k',N_k(X_k',X_k'-1))}
\\
&&\hspace*{47pt}{} +
\frac{\widetilde{w}(Z_k(X_k'+1),\ldots,Z_k(X_k'+n_0))}{\widetilde{f}^+(X_k',N_k(X_k',X_k'+1))}\biggr),
\end{eqnarray*}
for some $n_0\in\mathbb{N}$ and for any $k\in\mathbb{N}$, where the functions $\widetilde{f}{}^\pm$ are positive and
where $\widetilde{w}$ is positive and nondecreasing with respect to each variable. Note that the resulting term
$\widetilde{w}/\widetilde{f}{}^\pm$ can be nonmonotonic.

Finally, we refer the reader to the concluding remarks of \cite{ETW} that have been very useful to the author to
learn more about some open problems directly related to the model studied in the current paper.


\section*{Acknowledgments}
This work was done by the author during his PhD at the Universit\'e Paul Sabatier, Toulouse.

I would like to thank my PhD advisor Pierre Tarr\`es for introducing me to this problem, for motivating discussions and
for uncountably many comments on earlier versions of this work.



\printaddresses

\begin{thebibliography}{21}

\bibitem{APP}
\begin{barticle}[mr]
\bauthor{\bsnm{Amit},~\bfnm{Daniel~J.}\binits{D.~J.}},
\bauthor{\bsnm{Parisi},~\bfnm{G.}\binits{G.}} \AND
\bauthor{\bsnm{Peliti},~\bfnm{L.}\binits{L.}}
(\byear{1983}).
\btitle{Asymptotic behavior of the ``true'' self-avoiding walk}.
\bjournal{Phys. Rev. B (3)}
\bvolume{27}
\bpages{1635--1645}.
\bid{issn={0163-1829}, mr={0690540}}
\end{barticle}
%

\bptok{imsref}%
\endbibitem

\bibitem{BSS1}
\begin{barticle}[mr]
\bauthor{\bsnm{Basdevant},~\bfnm{Anne-Laure}\binits{A.-L.}},
\bauthor{\bsnm{Schapira},~\bfnm{Bruno}\binits{B.}} \AND
\bauthor{\bsnm{Singh},~\bfnm{Arvind}\binits{A.}}
(\byear{2014}).
\btitle{Localization on 4 sites for vertex-reinforced random walks on {$\Bbb{Z}$}}.
\bjournal{Ann. Probab.}
\bvolume{42}
\bpages{527--558}.
\bid{doi={10.1214/12-AOP811}, issn={0091-1798}, mr={3178466}}
\end{barticle}
%

\bptok{imsref}%
\endbibitem

\bibitem{BSS2}
\begin{barticle}[mr]
\bauthor{\bsnm{Basdevant},~\bfnm{Anne-Laure}\binits{A.-L.}},
\bauthor{\bsnm{Schapira},~\bfnm{Bruno}\binits{B.}} \AND
\bauthor{\bsnm{Singh},~\bfnm{Arvind}\binits{A.}}
(\byear{2014}).
\btitle{Localization of a vertex reinforced random walk on {$\Bbb{Z}$} with sub-linear weight}.
\bjournal{Probab. Theory Related Fields}
\bvolume{159}
\bpages{75--115}.
\bid{doi={10.1007/s00440-013-0502-3}, issn={0178-8051}, mr={3201918}}
\end{barticle}
%

\bptok{imsref}%
\endbibitem

\bibitem{12086375}
\begin{barticle}[mr]
\bauthor{\bsnm{Benaim},~\bfnm{Michel}\binits{M.}},
\bauthor{\bsnm{Raimond},~\bfnm{Olivier}\binits{O.}} \AND
\bauthor{\bsnm{Schapira},~\bfnm{Bruno}\binits{B.}}
(\byear{2013}).
\btitle{Strongly vertex-reinforced-random-walk on a complete graph}.
\bjournal{ALEA Lat. Am. J. Probab. Math. Stat.}
\bvolume{10}
\bpages{767--782}.
\bid{issn={1980-0436}, mr={3125746}}
\end{barticle}
%

\bptok{imsref}%
\endbibitem

\bibitem{Davis}
\begin{barticle}[mr]
\bauthor{\bsnm{Davis},~\bfnm{Burgess}\binits{B.}}
(\byear{1990}).
\btitle{Reinforced random walk}.
\bjournal{Probab. Theory Related Fields}
\bvolume{84}
\bpages{203--229}.
\bid{doi={10.1007/BF01197845}, issn={0178-8051}, mr={1030727}}
\end{barticle}
%

\bptok{imsref}%
\endbibitem

\bibitem{LD}
\begin{barticle}[mr]
\bauthor{\bsnm{Dumaz},~\bfnm{Laure}\binits{L.}}
(\byear{2012}).
\btitle{A clever (self-repelling) burglar}.
\bjournal{Electron. J. Probab.}
\bvolume{17}
\bpages{no. 61, 17}.
\bid{doi={10.1214/EJP.v17-1758}, issn={1083-6489}, mr={2959067}}
\end{barticle}
%

\bptok{imsref}%
\endbibitem

\bibitem{LDBT}
\begin{barticle}[mr]
\bauthor{\bsnm{Dumaz},~\bfnm{Laure}\binits{L.}} \AND
\bauthor{\bsnm{T{\'o}th},~\bfnm{B{\'a}lint}\binits{B.}}
(\byear{2013}).
\btitle{Marginal densities of the ``true'' self-repelling motion}.
\bjournal{Stochastic Process. Appl.}
\bvolume{123}
\bpages{1454--1471}.
\bid{doi={10.1016/j.spa.2012.11.011}, issn={0304-4149}, mr={3016229}}
\end{barticle}
%

\bptok{imsref}%
\endbibitem

\bibitem{Durrett}
\begin{bbook}[mr]
\bauthor{\bsnm{Durrett},~\bfnm{Rick}\binits{R.}}
(\byear{2010}).
\btitle{Probability: Theory and Examples},
\bedition{4th} ed.
\bpublisher{Cambridge Univ. Press},
\blocation{Cambridge}.
\bid{doi={10.1017/CBO9780511779398}, mr={2722836}}
\end{bbook}
%

\bptok{imsref}%
\endbibitem

\bibitem{ETW}
\begin{barticle}[mr]
\bauthor{\bsnm{Erschler},~\bfnm{Anna}\binits{A.}},
\bauthor{\bsnm{T{\'o}th},~\bfnm{B{\'a}lint}\binits{B.}} \AND
\bauthor{\bsnm{Werner},~\bfnm{Wendelin}\binits{W.}}
(\byear{2012}).
\btitle{Stuck walks}.
\bjournal{Probab. Theory Related Fields}
\bvolume{154}
\bpages{149--163}.
\bid{doi={10.1007/s00440-011-0365-4}, issn={0178-8051}, mr={2981420}}
\end{barticle}
%

\bptok{imsref}%
\endbibitem

\bibitem{ETW2}
\begin{bincollection}[author]
\bauthor{\bsnm{Erschler},~\bfnm{Anna}\binits{A.}},
\bauthor{\bsnm{T{\'o}th},~\bfnm{B{\'o}alint}\binits{B.}} \AND
\bauthor{\bsnm{Werner},~\bfnm{Wendelin}\binits{W.}}
(\byear{2012}).
\btitle{Some locally self-interacting walks on the integers}.
In \bbooktitle{Probability in Complex Physical Systems}
(\beditor{\bfnm{Jean-Dominique}\binits{J.-D.}~\bsnm{Deuschel}},
\beditor{\bfnm{Barbara}\binits{B.}~\bsnm{Gentz}},
\beditor{\bfnm{Wolfgang}\binits{W.}~\bsnm{K{\"o}nig}},
\beditor{\bfnm{Max}\binits{M.}~\bparticle{von}~\bsnm{Renesse}},
\beditor{\bfnm{Michael}\binits{M.}~\bsnm{Scheutzow}} \AND
\beditor{\bfnm{Uwe}\binits{U.}~\bsnm{Schmock}}, eds.).
\bseries{Springer Proceedings in Mathematics}
\bvolume{11}
\bpages{313--338}.
\bpublisher{Springer},
\blocation{Berlin}.
\end{bincollection}
%

\bptok{imsref}%
\endbibitem

\bibitem{Syst}
\begin{bbook}[mr]
\bauthor{\bsnm{Greene},~\bfnm{Daniel~H.}\binits{D.~H.}} \AND
\bauthor{\bsnm{Knuth},~\bfnm{Donald~E.}\binits{D.~E.}}
(\byear{2008}).
\btitle{Mathematics for the Analysis of Algorithms}.
\bseries{Modern Birkh\"auser Classics}.
\bpublisher{Birkh\"auser},
\blocation{Boston, MA}.
\bnote{Reprint of the third (1990) edition}.
\bid{mr={2381155}}
\end{bbook}
%

\bptok{imsref}%
\endbibitem

\bibitem{VLPT08}
\begin{bincollection}[mr]
\bauthor{\bsnm{Limic},~\bfnm{Vlada}\binits{V.}} \AND
\bauthor{\bsnm{Tarr{\`e}s},~\bfnm{Pierre}\binits{P.}}
(\byear{2008}).
\btitle{What is the difference between a square and a triangle?}
In \bbooktitle{In and Out of Equilibrium. 2}.
\bseries{Progress in Probability}
\bvolume{60}
\bpages{481--495}.
\bpublisher{Birkh\"auser},
\blocation{Basel}.
\bid{doi={10.1007/978-3-7643-8786-0_22}, mr={2477395}}
\end{bincollection}
%

\bptok{imsref}%
\endbibitem

\bibitem{Neveu}
\begin{bbook}[mr]
\bauthor{\bsnm{Neveu},~\bfnm{J.}\binits{J.}}
(\byear{1975}).
\btitle{Discrete-Parameter Martingales},
\bedition{Revised} ed.
\bpublisher{North-Holland},
\blocation{Amsterdam}.
\bnote{Translated from the French by T. P. Speed, North-Holland Mathematical Library, Vol. 10}.
\bid{mr={0402915}}
\end{bbook}
%

\bptok{imsref}%
\endbibitem

\bibitem{MR2282181}
\begin{barticle}[mr]
\bauthor{\bsnm{Pemantle},~\bfnm{Robin}\binits{R.}}
(\byear{2007}).
\btitle{A survey of random processes with reinforcement}.
\bjournal{Probab. Surv.}
\bvolume{4}
\bpages{1--79}.
\bid{doi={10.1214/07-PS094}, issn={1549-5787}, mr={2282181}}
\end{barticle}
%

\bptok{imsref}%
\endbibitem

\bibitem{PV}
\begin{barticle}[mr]
\bauthor{\bsnm{Pemantle},~\bfnm{Robin}\binits{R.}} \AND
\bauthor{\bsnm{Volkov},~\bfnm{Stanislav}\binits{S.}}
(\byear{1999}).
\btitle{Vertex-reinforced random walk on {${\bf Z}$} has finite range}.
\bjournal{Ann. Probab.}
\bvolume{27}
\bpages{1368--1388}.
\bid{doi={10.1214/aop/1022677452}, issn={0091-1798}, mr={1733153}}
\end{barticle}
%

\bptok{imsref}%
\endbibitem

\bibitem{Sellke}
\begin{barticle}[mr]
\bauthor{\bsnm{Sellke},~\bfnm{T.}\binits{T.}}
(\byear{2008}).
\btitle{Reinforced random walk on the {$d$}-dimensional integer lattice}.
\bjournal{Markov Process. Related Fields}
\bvolume{14}
\bpages{291--308}.
\bid{issn={1024-2953}, mr={2437533}}
\end{barticle}
%

\bptok{imsref}%
\endbibitem

\bibitem{PT5pts}
\begin{barticle}[mr]
\bauthor{\bsnm{Tarr{\`e}s},~\bfnm{Pierre}\binits{P.}}
(\byear{2004}).
\btitle{Vertex-reinforced random walk on {$\Bbb Z$} eventually gets stuck on five points}.
\bjournal{Ann. Probab.}
\bvolume{32}
\bpages{2650--2701}.
\bid{doi={10.1214/009117907000000694}, issn={0091-1798}, mr={2078554}}
\end{barticle}
%

\bptok{imsref}%
\endbibitem

\bibitem{PTsurvey}
\begin{bmisc}[author]
\bauthor{\bsnm{Tarr{\`e}s},~\bfnm{Pierre}\binits{P.}}
(\byear{2011}).
\bhowpublished{Localization of reinforced random walks.
Available at \arxivurl{ArXiv:1103.5536}.}
\end{bmisc}
%

\bptok{imsref}%
\endbibitem

\bibitem{BT9}
\begin{barticle}[mr]
\bauthor{\bsnm{T{\'o}th},~\bfnm{B{\'a}lint}\binits{B.}}
(\byear{1995}).
\btitle{The ``true'' self-avoiding walk with bond repulsion on {$\bold Z$}: Limit theorems}.
\bjournal{Ann. Probab.}
\bvolume{23}
\bpages{1523--1556}.
\bid{issn={0091-1798}, mr={1379158}}
\end{barticle}
%

\bptok{imsref}%
\endbibitem

\bibitem{BT10}
\begin{bincollection}[mr]
\bauthor{\bsnm{T{\'o}th},~\bfnm{B{\'a}lint}\binits{B.}}
(\byear{2001}).
\btitle{Self-interacting random motions}.
In \bbooktitle{European {C}ongress of {M}athematics, {V}ol. I ({B}arcelona, 2000)}.
\bseries{Progr. Math.}
\bvolume{201}
\bpages{555--564}.
\bpublisher{Birkh\"auser},
\blocation{Basel}.
\bid{mr={1905343}}
\end{bincollection}
%

\bptok{imsref}%
\endbibitem

\bibitem{BTWW}
\begin{barticle}[mr]
\bauthor{\bsnm{T{\'o}th},~\bfnm{B{\'a}lint}\binits{B.}} \AND
\bauthor{\bsnm{Werner},~\bfnm{Wendelin}\binits{W.}}
(\byear{1998}).
\btitle{The true self-repelling motion}.
\bjournal{Probab. Theory Related Fields}
\bvolume{111}
\bpages{375--452}.
\bid{doi={10.1007/s004400050172}, issn={0178-8051}, mr={1640799}}
\end{barticle}
%

\bptok{imsref}%
\endbibitem

\end{thebibliography}
\end{document}